\newif\ifarXiv
 \arXivtrue 

\documentclass[12pt,reqno]{amsart}

\usepackage{amsmath , amssymb, latexsym }
\usepackage{graphics}
\usepackage{graphicx}
\usepackage{hyperref} 

\usepackage{amsmath,amsthm,amsopn,amssymb}

\newtheorem{thm}{Theorem}[section]
\newtheorem*{Janson}{Janson's inequality}
\newtheorem*{Chernoff}{Chernoff's inequality}
\newtheorem*{FKG}{The FKG inequality}
\newtheorem{lemma}[thm]{Lemma}
\newtheorem{cor}[thm]{Corollary}
\newtheorem{prop}[thm]{Proposition}
\newtheorem*{claim}{Claim}
\newtheorem*{main_ass}{Main assumption}

\theoremstyle{definition}

\newtheorem*{defn}{Definition}
\newtheorem{cl}{Claim}
\newtheorem{remark}[thm]{Remark}

\makeatletter
\@addtoreset{cl}{section}
\makeatother

\DeclareMathOperator{\SF}{SF_0}
\DeclareMathOperator{\Var}{Var}

\newcommand{\one}{\mathbf{1}}
\newcommand{\Ex}{\mathbb{E}}
\newcommand{\G}{\mathcal{G1}}

\newcommand{\N}{\mathbb{N}}
\newcommand{\ZZ}{\mathbb{Z}}

\newcommand{\A}{\mathcal{A}}

\newcommand{\ds}{\displaystyle}

\def\G{\mathcal{G}}

\def\P{\mathcal{P}}

\def\S{\mathcal{S}}

\def\V{\mathcal{V}}

\def\Ex{\mathbb{E}}

\def\Pr{\mathbb{P}}

\def\RR{\mathbb{R}}
\def\ZZ{\mathbb{Z}}

\def\le{\leqslant}
\def\ge{\geqslant}

\def\eps{\varepsilon}
\def\<{\langle}
\def\>{\rangle}

\title{Random sum-free subsets of Abelian groups}
\author{J\'ozsef Balogh}
\address{Department of Mathematics, University of Illinois, 1409 W. Green Street, Urbana, IL 61801; and Department of Mathematics, University of California, San Diego, La Jolla, CA 92093} \email{jobal@math.uiuc.edu}

\author{Robert Morris} 
\address{IMPA, Estrada Dona Castorina 110, Jardim Bot\^anico, Rio de Janeiro, RJ, Brasil} \email{rob@impa.br}

\author{Wojciech Samotij}
\address{School of Mathematical Sciences, Tel Aviv University, Tel Aviv 69978, Israel; and Trinity College, Cambridge CB2 1TQ, UK} \email{samotij@post.tau.ac.il}

\thanks{Research supported in part by: (JB) NSF CAREER Grant DMS-0745185, UIUC Campus Research Board Grants 09072 and 11067, and OTKA Grant K76099; (RM) ERC Advanced grant DMMCA, and a Research Fellowship from Murray Edwards College, Cambridge; (WS) Parker Fellowship and Schark Fellowship (from UIUC Mathematics Department) and ERC Advanced Grant DMMCA}

\begin{document}

\begin{abstract}
  We characterize the structure of maximum-size sum-free subsets of a random subset of an Abelian group $G$. In particular, we determine the threshold above which, with high probability as $|G| \to \infty$, each such subset is contained in some maximum-size sum-free subset of $G$, whenever $q$ divides $|G|$ for some (fixed) prime $q$ with $q \equiv 2 \pmod 3$. Moreover, in the special case $G = \ZZ_{2n}$, we determine the sharp threshold for the above property. The proof uses recent `transference' theorems of Conlon and Gowers, together with stability theorems for sum-free subsets of Abelian groups.
\end{abstract}

\maketitle

\section{Introduction}

One of the most important developments in Combinatorics over the past twenty years has been the introduction and proof of various `random analogues' of well-known theorems in Extremal Graph Theory, Ramsey Theory, and Additive Combinatorics. Such questions were first introduced for graphs by Babai, Simonovits, and Spencer~\cite{BSS} (see also the work of Frankl and R{\"o}dl~\cite{FrRo}) and for additive structures by Kohayakawa, {\L}uczak, and R\"odl~\cite{KLR1}. There has since been a tremendous interest in such problems (see, for example, \cite{FRRT,Luc,RR1,RR2}). This extensive study has recently culminated in the remarkable results of Conlon and Gowers~\cite{CG} and Schacht~\cite{Sch} (see also \cite{BaMoSa-ind, FRS, Sa, SaTh}) in which a general theory was developed to attack such questions.

The main theorems in~\cite{CG} and~\cite{Sch} resolved many long-standing open questions; however, they provide only asymptotic results. For example, they imply that, with high probability\footnote{We say that a sequence $(\A_n)$ of events holds with \emph{high probability} if the probability that $\A_n$ holds tends to $1$ as $n \to \infty$.}, the largest triangle-free subgraph of the random graph $G(n,p)$ has 
\[
\left( \frac{1}{2} + o(1) \right) p {n \choose 2}
\]
edges if $p \gg 1/ \sqrt{n}$, which is best possible. A much more precise question asks the following: For which functions $p = p(n)$ is, with high probability, the largest triangle-free subgraph of $G(n,p)$ bipartite? It was proved that this is true if $p = 1/2$ by Erd{\H{o}}s, Kleitman, and Rothschild~\cite{EKR}, if $p \ge 1/2 - \delta$ by Babai, Simonovits, and Spencer~\cite{BSS}, and if $p \ge n^{-\eps}$ by Brightwell, Panagiotou, and Steger~\cite{BPS}; here, $\delta$ and $\eps$ are small positive constants. Finally, after this work had been completed, DeMarco and Kahn~\cite{DMKa} showed that this also holds if $p \ge C \sqrt{\log n / n}$, where $C$ is some positive constant. This is best possible up to the value of $C$ since, as pointed out in~\cite{BPS}, the statement is false when $p \le \frac{1}{10} \sqrt{\log n / n}$.

In the setting of additive number theory, the first results on such problems were obtained by Kohayakawa, {\L}uczak, and R\"odl~\cite{KLR1}, who proved the following random version of Roth's~Theorem~\cite{Roth}. The \emph{$p$-random subset} of a set $X$ is the random subset of $X$, where each element is included with probability $p$, independently of all other elements. We say that a subset $B$ of an Abelian group $G$ is \emph{$\delta$-Roth} if every subset $A \subseteq B$ with $|A| \ge \delta |B|$ contains a $3$-term arithmetic progression. The main result of~\cite{KLR1} is that for every $\delta > 0$, if $p \ge C(\delta) / \sqrt{n}$, then the $p$-random subset of $\ZZ_n$ is $\delta$-Roth with high probability. This is again best possible up to the constant $C(\delta)$.

Given a sequence $(G_n)$ of Abelian groups with $|G_n| = n$, we say that $p_c = p_c(n)$ is a \emph{threshold function} for a property $\P$ if the $p$-random subset $B \subseteq G_n$ has $\P$ with high probability if $p = p(n) \gg p_c(n)$ and $B$ does not have $\P$ with high probability if $p \ll p_c$. A threshold function is \emph{sharp} if moreover the above holds with $p > (1 + \eps)p_c$ and $p < (1 - \eps)p_c$ for any fixed $\eps > 0$. Bollob\'as and Thomason~\cite{BT} proved that every monotone property has a threshold function, and there exists a large body of work addressing the problem of the existence of sharp thresholds for monotone properties, see~\cite{Fried,FK,Hatami}. Non-monotone properties, such as the one which we shall be studying, are more complicated and in general they do not have such thresholds. We shall prove that such a threshold \emph{does} exist in the setting of sum-free subsets of Abelian groups described below and, moreover, we shall determine it. For $\ZZ_{2n}$, we shall prove much more: that there exists a sharp threshold.

A subset $A$ of an Abelian group $G$ is said to be \emph{sum-free} if there is no solution to the equation $x + y = z$ with $x,y,z \in A$. Such forbidden triples $(x,y,z)$ are called \emph{Schur triples}. (Note that we forbid some triples with $x = y$, and also some with $x = z$; the results and proofs in the case that such triples are allowed are identical.) In 1916, Schur~\cite{Schur} proved that, given any $r$-colouring of the integers, there exists a monochromatic Schur triple. Graham, R\"odl, and Ruci\'nski~\cite{GRR} studied the random version of Schur's Theorem and proved that the threshold function for the existence of a 2-colouring of the $p$-random set $B \subseteq \ZZ_n$ without a monochromatic Schur triple is $1/\sqrt{n}$. The extremal version of this question, that is, that of determining the size of the largest sum-free subset of the $p$-random subset of $\ZZ_n$, was open for fifteen years, until it was recently resolved by Conlon and Gowers~\cite{CG} and Schacht~\cite{Sch}.  

Sum-free sets have been extensively studied over the past 40 years (see, e.g.,~\cite{AW,AK,Cam2,CKS,Frei,Yap69}), mostly in the extremal setting. For example, in 1969 Diananda and Yap~\cite{DY} determined the extremal density for a sum-free set in every Abelian group $G$ such that $|G|$ has a prime divisor $q$ with $q \not\equiv 1 \pmod 3$. However, more than 30 years had passed until the classification was completed by Green and Ruzsa~\cite{GR2} (see Theorem~\ref{thm:muG}, below). Another well-studied problem was the Cameron-Erd\H{o}s Conjecture (see~\cite{A,C,CE1,GR1,LS}), which asked for the number of sum-free subsets of the set $\{1, \ldots, n\}$ and was finally solved by Green~\cite{G1} and, independently, by Sapozhenko~\cite{Sa-CE}. A refinement of their theorem has recently been established in~\cite{AlBaMoSa-CE}.

In this paper, we shall study the analogue of the Babai-Simonovits-Spencer problem for sum-free sets. Let $G = \ZZ_{2n}$. It is not hard to see that the unique maximum-size sum-free set in $G$ is the set $O_{2n}$ of odd numbers. We shall prove a probabilistic version of this statement and determine the threshold $p_c$ for the property that the unique maximum-size sum-free subset of the $p$-random subset $B$ of $G$ is the set $B \cap O_{2n}$. In fact, we shall do better and determine a \emph{sharp} threshold for this property. 

For a set $B \subseteq G$, let $\SF(B)$ denote the collection of maximum-size sum-free subsets of~$B$. The following theorem is our main result.

\begin{thm}
  \label{thm:sharp}
  Let $C = C(n) \gg \log n / n$ and let $p = p(n)$ satisfy $p^2 n = C \log n$. Let $G_p$ be the $p$-random subset of $\ZZ_{2n}$. Then
  \[
  \Pr\Big( \SF(G_p) = \{G_p \cap O_{2n}\} \Big) \to
    \begin{cases}
      0 & \text{if $\limsup_n C(n) < 1/3$} \\
      1 & \text{if $\liminf_n C(n) > 1/3$}
    \end{cases}
  \]
  as $n \to \infty$. Moreover, if $C = 1/3$, then $\Pr\big( \SF(G_p) = \{G_p \cap O_{2n}\} \big) \to 0$ as $n \to \infty$.
\end{thm}

We remark that the threshold for the asymptotic versions of this statement, i.e., that (i) the largest sum-free subset of $G_p$ has $(1+o(1))pn$ elements and (ii) that all sum-free subsets of $G_p$ with $(1+o(1))pn$ elements contain only $o(pn)$ even numbers, determined by Conlon and Gowers~\cite{CG} (both (i) and (ii)) and Schacht~\cite{Sch} (only (i)), is $1/\sqrt{n}$ and so it differs from our threshold by a factor of $\sqrt{\log n}$. 

We shall also determine the threshold $p_c(n)$ for all Abelian groups of type I($q$) (see the definition below), i.e., those for which $|G| = n$ has a (fixed) prime divisor $q$ with $q \equiv 2 \pmod 3$. Say that a set $B \subseteq G$ is \emph{sum-free good} if every maximum-size sum-free subset of $B$ is of the form $B \cap A$ for some $A \in \SF(G)$.

\begin{thm}
  \label{thm:main}
  Let $q$ be a prime number satisfying $q \equiv 2 \pmod 3$. There exist positive constants $c_q$ and $C_q$ such that the following holds.

  Let $\G = (G_n)_{n \in q\N}$ be a sequence of Abelian groups, with $|G_n| = n$, such that $q$ divides $|G|$ for every $G \in \G$. Let $C = C(n) \gg \log n / n$, let $p = p(n)$ satisfy $p^2 n = C \log n$, and let $G_p$ be the $p$-random subset of $G = G_n$. Then 
  \[
  \Pr\Big( G_p \textup{ is sum-free good} \Big) \to
  \begin{cases}
    0 & \text{if $C < c_q$} \\
    1 & \text{if $C > C_q$}
  \end{cases}
  \]
  as $n \to \infty$.
\end{thm}

We remark that the conclusion of the theorem fails to hold when we do not assume that $|G|$ has a prime divisor $q$ with $q \equiv 2 \pmod 3$. Indeed, we shall show (see Proposition~\ref{c_ex_2}) that if $G = \ZZ_{3q}$, where $q$ is prime and $q \equiv 1 \pmod 3$, then the probability that $G_p$ is sum-free good goes to zero (as $n = |G| \to \infty$) for all $p \ll ( n \log n )^{-1/3}$. We shall also prove the same bound for the group $G = \ZZ_q$, where $q \equiv 2 \pmod 3$ (see Proposition~\ref{c_ex_1}), which shows that the condition $n \gg q$ in Theorem~\ref{thm:main} is also necessary.

We note also that, perhaps not surprisingly, the constant $C = 1/3$ in Theorem~\ref{thm:sharp} is not the same for every Abelian group $G$ with $|G|$ even. Indeed, we shall show (see Proposition~\ref{2^k}) that for the hypercube $G = \ZZ_2^k$ with $|G| = 2n$, the threshold $p_c$ is at least $\sqrt{\log n / (2n)}$. 

Finally, we remark that, in a recent paper with Alon~\cite{AlBaMoSa-SF}, we prove a `counting version' of Theorem~\ref{thm:main}. More precisely, we determine the threshold for $m$ above which a uniformly selected random $m$-element sum-free subset of an $n$-element group $G$ of type $I$ is contained in some maximum size sum-free subset of $G$. The somewhat related problem of determining the size of the largest Sidon set\footnote{A set $A$ is called a \emph{Sidon set} if all the sums $x + y$, where $x, y \in A$ with $x \le y$ are distinct.} in a $p$-random subset of the set $\{1, \ldots, n\}$ has also recently been addressed in~\cite{KoLeRoSa}, where sharp bounds were obtained for a large range of~$p$.

The remainder of the paper is organized as follows. In Section~\ref{sec:preliminaries}, we collect some extremal results on Abelian groups and the probabilistic tools that will be needed later. In particular, we state our main tool, a theorem of Conlon and Gowers~\cite{CG} (see also \cite{Sa,Sch}), which provides asymptotic versions of Theorems~\ref{thm:sharp} and~\ref{thm:main}. Theorem~\ref{thm:main} is proved in Section~\ref{sec:proof-thm-main} and in Section~\ref{sec:disc-thm-main} we prove the lower bounds for other Abelian groups described above. Finally, in Section~\ref{sec:proof-thm-Z2n}, we prove Theorem~\ref{thm:sharp}.

\section{Preliminaries and tools}\label{sec:preliminaries}

\subsection{Extremal results on Abelian groups}

Let $G$ be a finite Abelian group. Given two subsets $A, B \subseteq G$, we let
\[A
\pm B = \{ a \pm b \colon a \in A \text{ and } b \in B\}.
\]
Note that a subset $A$ of $G$ is sum-free if $(A+A) \cap A = \emptyset$. We begin with an important definition in the study of sum-free subsets of finite Abelian groups.

\begin{defn}
  Let $G$ be a finite Abelian group. We say that
  \begin{enumerate}
  \item
    $G$ is of type I if $|G|$ has at least one prime divisor $q$ with $q \equiv 2 \pmod 3$.
  \item
    $G$ is of type II if $|G|$ has no prime divisors $q$ with $q \equiv 2 \pmod 3$, but $|G|$ is divisible by $3$.
  \item
    $G$ is of type III if every prime divisor $q$ of $|G|$ satisfies $q \equiv 1 \pmod 3$.
  \end{enumerate}
  Moreover, we say that $G$ is of type I($q$) if $G$ is of type I and $q$ is the smallest prime divisor of $|G|$ with $q \equiv 2 \pmod 3$.
\end{defn}

Given an Abelian group $G$, let $\mu(G)$ be the density of the largest sum-free subset of $G$ (so that this subset has $\mu(G) |G|$ elements). As noted in the Introduction, the problem of determining $\mu(G)$ for an arbitrary $G$ has been studied for more than 40 years, but only recently was it solved completely by Green and Ruzsa~\cite{GR2}. 

\begin{thm}[Diananda and Yap~\cite{DY}, Green and Ruzsa~\cite{GR2}]
  \label{thm:muG}
  Let $G$ be an arbitrary finite Abelian group. Then
  \[
  \mu(G) =
  \begin{cases}
    \frac{1}{3} + \frac{1}{3q} & \text{if $G$ is of type I($q$),} \\
    \frac{1}{3} & \text{if $G$ is of type II,} \\
    \frac{1}{3} - \frac{1}{3m} & \text{if $G$ is of type III,} \\
  \end{cases}
  \]
  where $m$ is the exponent (the largest order of any element) of $G$.
\end{thm}

Note that it follows immediately from Theorem~\ref{thm:muG} that  
\[
2/7 \, \le \, \mu(G) \, \le \, 1/2
\]
for every finite Abelian group $G$.

Recall that $\SF(G)$ denotes the collection of all maximum-size sum-free subsets of $G$, i.e., those that have $\mu(G)|G|$ elements. As well as determining $\mu(G)$, Diananda and Yap~\cite{DY} described $\SF(G)$ for all groups of type I (see also~\cite[Lemma 5.6]{GR2}). 

\begin{thm}[Diananda and Yap~\cite{DY}]
  \label{thm:SFG-structure}
  Let $G$ be a group of type $I(q)$ for some prime $q = 3k + 2$. For each $A \in \SF(G)$, there exists a homomorphism $\varphi \colon G \to \ZZ_q$ such that $A = \varphi^{-1}(\{k+1, \ldots, 2k+1\})$.

  In other words, every $A \in \SF(G)$ is a union of cosets of some subgroup $H$ of $G$ of index $q$, $A/H$ is an arithmetic progression in $G/H$, and $A \cup (A+A) = G$. 
\end{thm}

We shall also need the following well-known bound on the number of homomorphisms from an arbitrary finite Abelian group to a cyclic group of prime order, which follows easily from Kronecker's Decomposition Theorem. 

\begin{prop}
  \label{prop:HomGZq}
  For every prime $q$, the number of homomorphisms from a finite Abelian group $G$ to the cyclic group $\ZZ_q$ is at most $|G|$.
\end{prop}

Theorem~\ref{thm:SFG-structure} and Proposition~\ref{prop:HomGZq} immediately imply the following.

\begin{cor}
  \label{cor:SFG}
  Let $G$ be an arbitrary group of type I. Then $|\SF(G)| \le |G|$.
\end{cor}

We will also need the following corollary from the classification of maximum-size sum-free subsets of $\ZZ_{3q}$, where $q$ is a prime with $q \equiv 1 \pmod 3$, due to Yap~\cite{Yap70}.

\begin{cor}
  \label{cor:SFZ3q}
  If $q$ is a prime with $q \equiv 1 \pmod 3$, then $|\SF(\ZZ_{3q})| \le 2q$.
\end{cor}

In the proof of Theorem~\ref{thm:main}, we will use the following two lemmas, which were proved in~\cite{GR2} and~\cite{LLS} respectively. The first lemma establishes a strong stability property for large sum-free subsets of groups of type I.

\begin{lemma}[Green and Ruzsa~\cite{GR2}]
  \label{lemma:stability}
  Let $G$ be an Abelian group of type $I(q)$. If $A$ is a sum-free subset of $G$ and 
  \[
  |A| \, \ge \, \left( \mu(G) - \frac{1}{3q^2+3q} \right)|G|,
  \]
  then $A$ is contained in some $A' \in \SF(G)$.
\end{lemma}

The second lemma is a rather straightforward corollary of a much stronger result of Lev, {\L}uczak, and Schoen~\cite{LLS}.

\begin{lemma}[Lev, {\L}uczak, and Schoen~\cite{LLS}]
  \label{lemma:saturation}
  Let $\eps > 0$, let $G$ be a finite Abelian group, and let $A \subseteq G$. If 
  \[
  |A| \, \ge \, \left( \ds\frac{1}{3} + \eps \right)|G|,
  \]
  then one of the following holds:
  \begin{itemize}
  \item[$(a)$] $|A \setminus A'| \le \eps|G|$ for some sum-free $A' \subseteq A$.\smallskip
  \item[$(b)$] $A$ contains at least $\eps^3|G|^2/27$ Schur triples.
  \end{itemize}
\end{lemma}

We remark that a more general statement, the so-called Removal Lemma for groups, was proved by Green~\cite{G2} (for Abelian groups) and Kr{\'a}l, Serra, and Vena~\cite{KSV} (for arbitrary groups). Lemmas~\ref{lemma:stability} and \ref{lemma:saturation} immediately imply the following.

\begin{cor}\label{cor:saturation}
  Let $G$ be a group of type $I(q)$, where $q$ is a prime $q \equiv 2 \pmod 3$, and let $0 < \eps < \frac{1}{9q^2+9q}$. Let $A \subseteq G$ and suppose that 
  \[
  |A| \, \ge \, \big( \mu(G) - \eps \big) |G|.
  \]
  Then one of the following holds: 
  \begin{itemize}
  \item[$(a)$] $|A \setminus A'| \le \eps|G|$ for some $A' \in \SF(G)$.\smallskip
  \item[$(b)$] $A$ contains at least $\eps^3|G|^2/27$ Schur triples.
  \end{itemize}
\end{cor}

\begin{proof}
  Suppose first that $|A \setminus A''| > \eps|G|$ for every sum-free $A'' \subseteq A$. By Theorem~\ref{thm:muG} and our choice of $\eps$, we have $|A| \ge (1/3 + \eps)|G|$. Hence, by Lemma~\ref{lemma:saturation}, $A$ contains at least $\eps^3|G|^2/27$ Schur triples, as required.

  So assume that there exists a sum-free set $A'' \subseteq A$ with $|A \setminus A''| \le \eps|G|$. Then
  \[
  |A''| \, \ge \, |A| - \eps |G| \, \ge \, \big( \mu(G) - 2\eps \big) |G| \, > \, \left( \mu(G) - \frac{1}{3q^2+3q} \right) |G|,
  \]
  and so, by Lemma~\ref{lemma:stability}, $A''$ is contained in some $A' \in \SF(G)$. But then 
  \[
  |A \setminus A'| \, \le \, |A \setminus A''| \, \le \, \eps|G|,
  \]
  and so we are done in this case as well.
\end{proof}

\subsection{Probabilistic tools}

We shall next recall three well-known probabilistic inequalities: the FKG inequality, Janson's inequality, and Chernoff's inequality.

Given an arbitrary set $X$ and a real number $p \in [0,1]$, we denote by $X_p$ the $p$-random subset of $X$, i.e., the random subset of $X$, where each element is included with probability $p$ independently of all other elements. In the proof of our main results, we shall need several bounds on the probabilities of events of the form 
\[
\bigwedge_{i \in I} (B_i \nsubseteq X_p),
\]
where $B_i$ are subsets of $X$. The first such estimate can be easily derived from the FKG Inequality (see, e.g., \cite[Section 6.3]{AS}).

\begin{FKG}
  \label{fkg}
  Suppose that $\{B_i\}_{i \in I}$ is a family of subsets of a finite set $X$ and let $p \in [0,1]$. Then
  \begin{equation}
    \label{ineq:correlation}
    \Pr\big( B_i \nsubseteq X_p \text{ for all $i \in I$} \big) \, \ge \, \prod_{i \in I} \Pr\big( B_i \nsubseteq X_p \big) \, = \, \prod_{i \in I} \big( 1 - p^{|B_i|} \big).
  \end{equation}
\end{FKG}

The second result, due to Janson (see, e.g., \cite[Section 8.1]{AS}), gives an upper bound on the probability in the left-hand side of~\eqref{ineq:correlation} expressed in terms of the intersection pattern of the sets $B_i$.

\begin{Janson}
  \label{janson}
  Suppose that $\{B_i\}_{i \in I}$ is a family of subsets of a finite set $X$ and let $p \in [0,1]$. Let
  \[
  M = \prod_{i \in I} \big( 1 - p^{|B_i|} \big), \quad
  \mu = \sum_{i \in I} p^{|B_i|}, \quad \text{and} \quad
  \Delta = \sum_{i \sim j} p^{|B_i \cup B_j|},
  \]
  where $i \sim j$ denotes the fact that $i \neq j$ and $B_i \cap B_j \neq \emptyset$. Then,
  \[
  \Pr(B_i \nsubseteq X_p \text{ for all $i \in I$}) \le \min\left\{M e^{\Delta/(2-2p)}, e^{-\mu + \Delta/2}\right\}.
  \]
  Furthermore, if $\Delta \ge \mu$, then
  \[
  \Pr(B_i \nsubseteq X_p \text{ for all $i \in I$}) \le e^{-\mu^2/(2\Delta)}.
  \]
\end{Janson}

Finally, we will need the following well-known concentration result for binomial random variables; see, e.g., \cite[Appendix A]{AS}.

\begin{Chernoff}
  \label{chernoff}
  Let $X$ be the binomial random variable with parameters $n$ and $p$. Then for every positive $a$,
  \[
  \Pr\big( X - pn > a \big) \, < \, \exp\left( -\frac{a^2}{2pn} \,+\, \frac{a^3}{2(pn)^2} \right) \quad \text{and} \quad \Pr\big( X - pn < -a \big) \,<\, \exp\left( -\frac{a^2}{2pn} \right).
  \]
\end{Chernoff}

\subsection{The Conlon-Gowers Theorem}

In order to prove the $1$-statements in Theorems~\ref{thm:sharp} and~\ref{thm:main}, i.e., that every maximum-size sum-free subset of $G_p$ is of the form $A \cap G_p$ for some $A \in \SF(G)$, provided that $p$ is sufficiently large, we need the following approximate version of this statement, which can be derived from Corollary~\ref{cor:saturation} using the transference theorems of Conlon and Gowers~\cite{CG}.
\ifarXiv
This is done in Section~\ref{CGSsec}.
\fi

\begin{thm}
  \label{thm:approx-stability-full}
  For every $\eps > 0$ and every prime $q$ with $q \equiv 2 \pmod 3$, there exists a constant $C = C(q, \eps) > 0$ such that the following holds. Let $G$ be an arbitrary $n$-element group of type $I(q)$. If 
  \[
  p \, \ge \, \frac{C}{\sqrt{n}},
  \]
  then, with high probability as $n \to \infty$, for every sum-free subset of $G_p$ with 
  \[
  |B| \, \ge \, \left( \mu(G) - \frac{1}{40q^2+40q} \right) p|G|,
  \]
  there is an $A \in \SF(G)$ such that $|B \setminus A| \le \eps pn$.
\end{thm}

After this work had been completed a more transparent proof of Theorem~\ref{thm:approx-stability-full} than the one originally included here was given in \cite{Sa} using a refinement of the methods of \cite{Sch}. Moreover, Theorem~\ref{thm:approx-stability-full} can be quite easily deduced from the main results of \cite{BaMoSa-ind} and \cite{SaTh}.

\subsection{Notation} 

For the sake of brevity, we shall write $y$ for the set $\{y\}$.  We shall also use $\Delta$ to denote the ``correlation measure'' as in Janson's inequality, above, and $\Delta(G)$ for the maximum degree in a graph $G$. We trust that neither of these will confuse the reader.

\ifarXiv

\section{The Conlon-Gowers Method}\label{CGSsec}

In this section, we derive Theorem~\ref{thm:approx-stability-full} from the transference theorems of Conlon and Gowers~\cite{CG}. Unfortunately, for technical reasons, their methods can only be applied under the additional assumption that $p \ge C(q,\eps)^{-1}$. Therefore, we will show how to derive Theorem~\ref{thm:approx-stability-full} from the following statement, which in turn can be proved using the aforementioned transference theorems.

\begin{thm}
  \label{thm:approx-stability}
  For every $\eps \in (0,1)$, and every prime $q$ with $q \equiv 2 \pmod 3$, there exists a constant $C = C(q, \eps) > 0$ such that the following holds. Let $G$ be an arbitrary $n$-element group of type $I(q)$. If 
  \[
  \frac{C}{\sqrt{n}} \, \le \, p \, \le \, \frac{1}{C},
  \]
  then a.a.s.~for every sum-free subset of $G_p$ with 
  \[
  |B| \, \ge \, \left( \mu(G) - \frac{1}{10q^2+10q} \right) p|G|,
  \]
  there is an $A \in \SF(G)$ such that $|B \setminus A| \le \eps pn$.
\end{thm}

In order to deduce Theorem~\ref{thm:approx-stability-full} from Theorem~\ref{thm:approx-stability}, we simply chop $p$ into sufficiently small pieces, apply Theorem~\ref{thm:approx-stability} to each piece, and then show that we obtain the same set $A \in \SF(G)$ for (almost) all of the pieces. 

\begin{proof}[Proof of Theorem~\ref{thm:approx-stability-full}]
  Clearly, it suffices to consider the case $p \ge C(q, \eps)^{-1}$. To this end, fix some $p \in (0,1]$, let $c_q = 1 / (10q^2 + 10q)$, let $\eps' = \eps c_q/4$, and let $M$ be the least positive integer satisfying $2p/M \le 1/C'$, where $C' = C(q, \eps')$ is the constant defined in the statement of Theorem~\ref{thm:approx-stability}. Assign to each $x \in G$ a number $t(x) \in [0,1]$ uniformly at random and for all $i \in [M]$, let
  \[
  G^i = \{ x \in G \colon (i-1)p/M \le t(x) \le ip/M \}.
  \]
  It is not hard to see that $G^1 \cup \ldots \cup G^M$ has the same distribution as $G_p$, that $G^i$ has the same distribution as $G_{p/M}$ for every $i \in [M]$, and that $G^i \cup G^j$ has the same distribution as $G_{2p/M}$ for all distinct $i, j \in [M]$. Since $M = O(1)$, a.a.s.\ the following statements hold simultaneously:
  \begin{enumerate}
  \item[$(i)$]
    For all $i \in [M]$, $G^i$ satisfies the assertion of Theorem~\ref{thm:approx-stability}.
  \item[$(ii)$]
    For all distinct $i, j \in [M]$, $G^i \cup G^j$ satisfies the assertion of Theorem~\ref{thm:approx-stability}.
  \item[$(iii)$]
    For all $i \in [M]$ and distinct $A, A' \in \SF(G)$, $|G^i \cap A| \le (\mu(G) + \eps')pn/M$ and
    \[
    |G^i \cap A \cap A'| \, \le \, \left( \mu(G) - \frac{1}{2q^2} \right) \frac{pn}{M}.
    \]
  \end{enumerate}
  To see that $(iii)$ holds, observe that $|A \cap A'| \le (\mu(G) - 1/q^2)n$ since, by Theorem~\ref{thm:SFG-structure}, $A \setminus A'$ is a union of cosets of some subgroup $H \cap H'$, where $H$ and $H'$ are subgroups of index $q$ (and hence $H \cap H'$ has index $q$ or $q^2$). Now, since $p = \Theta(1)$ and $|\SF(G)| \le n$, by Corollary~\ref{cor:SFG}, the result follows by Chernoff's inequality.

Now, let $B$ be a sum-free subset of $G^1 \cup \ldots \cup G^M$ with $|B| \ge (\mu(G) - \eps')pn$. For each $i \in [M]$, let $B^i = G^i \cap B$ and let
  \[
  I \, = \, \left\{ i \in [M] \,\colon\, |B^i| \ge \big( \mu(G) - c_q \big) \frac{pn}{M} \right\}.
  \]
  Suppose that $i \in I$. Since $B^i \subset B$ is sum-free, $(i)$ implies that $|B^i \setminus A^i| \le \eps'pn/M$ for some $A^i \in \SF(G)$. Moreover, for every distinct $i, j \in I$, the set $B^i \cup B^j \subset B$ is also sum-free and $|B^i \cup B^j| = |B^i| + |B^j| \ge (\mu(G)-c_q)2pn/M$. Thus, by $(ii)$, we have $|(B^i \cup B^j) \setminus A| \le 2\eps'pn/M$ for some $A \in \SF(G)$. In particular, it follows from $(iii)$ that $A^i = A^j = A$, since otherwise
  \[
  |G^i \cap A \cap A^i| \, \ge \, |B^i \cap A \cap A^i| \, \ge \, |B^i| - |B^i \setminus A^i| - |B^i \setminus A| \, \ge \, \Big( \mu(G) - c_q - 3\eps' \Big)\frac{pn}{M},
  \]
  which contradicts $(iii)$, since $c_q + 3\varepsilon' \le 1/(2q^2)$. 
  
Let $A$ be the unique maximum-size sum-free subset of $G$ satisfying $A^i = A$ for all $i \in I$; we claim that $|B \setminus A| \le \eps pn$. Indeed, by the definition of $A^i$ and $(iii)$,
  \[
  |B^i| \le |G^i \cap A| + |B^i \setminus A| \le (\mu(G) + 2\eps')pn/M
  \]
  and hence
  \[
  \big( \mu(G) - \eps' \big) pn \, \le \, |B| \, \le \, |I| \Big( \mu(G) + 2\eps' \Big) \frac{pn}{M} + \big( M - |I| \big) \Big( \mu(G) - c_q \Big)\frac{pn}{M},
  \]
  which implies that $|I| \ge (1 - \eps)M$. We conclude that
  \[
  |B \setminus A| \, \le \, \sum_{i \in I} |B^i \setminus A| + \sum_{i \not\in I} |B^i| \, \le \, |I| \cdot \frac{\eps'pn}{M} + \big( M - |I| \big) \mu(G) \cdot \frac{pn}{M} \, \le \, \eps pn,
  \]
as required.
\end{proof}

In the remainder of this section, we shall sketch the proof of Theorem~\ref{thm:approx-stability}. Let $G_p$ be the $p$-random subset of $G$. Following~\cite{CG}, we define the associated measure of $G_p$, denoted $\mu$, by $\mu = p^{-1} \cdot \chi_{G_p}$, i.e., $\mu(x) = p^{-1}$ if $x \in G_p$ and $\mu(x) = 0$ otherwise. Let $S$ be the collection of all Schur triples in $G$ and note that $|S| = \Theta(n^2)$. Moreover, let $\V = \SF(G) \cup \{G\}$ and note that (by Corollary~\ref{cor:SFG}) $|\V| \le n+1$. The transference theorem proved in~\cite{CG} asserts that a.a.s.~for every function $f \colon G \to \RR$ with $0 \le f \le \mu$, there is a function $g \colon G \to [0,1]$ such that
\begin{equation}
  \label{ineq:trans-triples}
  \Ex_{s \in S} f(s_1)f(s_2)f(s_3) \ge \Ex_{s \in S} g(s_1)g(s_2)g(s_3) - \eps
\end{equation}
and
\begin{equation}
  \label{ineq:trans-traces}
  \left| \sum_{x \in V} f(x) - \sum_{x \in V} g(x) \right| \le \eps |G| \quad \text{for all $V \in \V$.}
\end{equation}
Moreover, we may assume that $g$ takes values only in $\{0,1\}$, i.e., $g$ is the characteristic function of some subset of $G$; see, e.g., \cite[Corollary 9.7]{CG}. In particular, if we let $f$ to be $p^{-1}$ times the characteristic function of some subset of $G_p$, then we will see that for every $B \subseteq G_p$, there exists a $B' \subseteq G$ such that $|B' \cap V| \approx p|B \cap V|$ for every $V \in \V$ and the number of Schur triples in $B'$ is by at most $\eps n^2$ larger than $p^{-3}$ times the number of Schur triples in $B$. Hence, if $B \subseteq G_p$ is sum-free and $|B| \ge (\mu(G) - c_q)pn$ for some small positive constant $c_q$, then the corresponding set $B'$ has at least $(\mu(G) - c_q - \eps)n$ elements and at most $\eps n^2$ Schur triples. By Corollary~\ref{cor:saturation}, there is an $A \in \SF(G)$ such that $|B' \setminus A| \le 3\eps^{1/3}n$. Finally, since $A \in \V$, we can conclude that $|B \setminus A| = O(\eps^{1/3}pn)$.

To complete the proof of Theorem~\ref{thm:approx-stability}, we still need to argue how the results of Conlon and Gowers~\cite{CG} imply that a.a.s.~every function $f$ with $0 \le f \le \mu$ can be approximated in the above sense by some $g \colon G \to \{0,1\}$. This implication would be a direct consequence of~\cite[Corollary 9.7]{CG} if $S$, the collection of Schur triples in $G$, was a so called {\em good system}, i.e., if $S$ satisfied the following conditions:
\begin{enumerate}
  \item[$(i)$]
    No sequence in $S$ has repeated elements.
  \item[$(ii)$]
    $S$ is {\em homogeneous}, i.e., for every $j \in [3]$ and every $x \in G$, the set $\{s \in S \colon s_j = x\}$, denoted $S_j(x)$, has the same size.
  \item[$(iii)$]
    $S$ has {\em two degrees of freedom}, i.e., whenever some $s, t \in S$ satisfy $s_i = t_i$ for two distinct indices $i$, then $s = t$.
\end{enumerate}
Indeed, if $S$ satisfied $(i)$--$(iii)$, then we could simply apply~\cite[Corollary 9.7]{CG}, since for families with two degrees of freedom, the threshold for $p$ required for the transference theorem is $n^{-\gamma/2}$, where $\gamma$ is defined by $|S_j(x)| = n^\gamma$. Since in our case $\gamma = 1$, the threshold is at $n^{-1/2}$.

Unfortunately, in our setting some sequences in $S$ have repeated elements (Schur triples of the form $(x,x,2x)$, $(x,0,x)$, or $(0,x,x)$) and when we remove them, the new set $S$ is no longer homogeneous as, e.g., some $y \in G$ satisfy $y = x + x$ for many different $x$. To overcome this problem, let
\[
D = \big\{ y \in G \colon y = x + x \text{ for more than $\sqrt{n}$ different $x$} \big\}.
\]
Since there are only $n$ sums of the form $x+x$, it follows that $|D| \le \sqrt{n}$. Next, we let $X = G \setminus (D \cup \{0\})$ and, instead of working with the collection of all Schur triples in $G$, we define the new set $S$ as follows:
\[
S = \big\{ (x,y,x+y) \in X^3 \colon x \neq y \big\}.
\]
Since $0 \not\in X$, it follows that no triple in $S$ has repeated elements. Moreover, for all $j \in [3]$ and every $x \in X$,
\[
n - 2\sqrt{n} - 1 \le |S_j(x)| \le n,
\]
i.e., the new set $S$ satisfies $|S_j(x)| = (1+o(1))n$ for all $j \in [3]$ and $x \in X$.

Before being able to state the version of the transference theorem that we are actually going to use, we need to do some preparation. Recall that for every $j \in [3]$ and $x \in X$, we defined $S_j(x) = \{s \in S \colon s_j = x\}$. Following~\cite{CG}, given functions $h_1, h_2, h_3 \colon X \to \RR$ and $j \in [3]$, we define their {\em $j$th convolution} $*_j(h_1, h_2, h_3)$ by
\begin{align*}
  *_j(h_1, h_2, h_3)(x) & = \frac{|X|}{|S|}\sum_{s \in S_j(x)} h_1(s_1) \ldots h_{j-1}(s_{j-1}) h_{j+1}(s_{j+1})h_3(s_3) \\
  & = \frac{|X||S_j(x)|}{|S|}\Ex_{s \in S_j(x)} h_1(s_1) \ldots h_{j-1}(s_{j-1}) h_{j+1}(s_{j+1})h_3(s_3).
\end{align*}
Moreover, we define an inner product of real-valued functions on $X$ by
\[
\langle f, g \rangle = \frac{1}{|X|}\sum_{x \in X}f(x)g(x).
\]
A crucial observation is that our (slightly modified in comparison with \cite[Definition 3.2]{CG}) definition of the convolutions $*_j$ guarantees that for each $j$,
\[
\langle h_j, *_j(h_1, h_2, h_3) \rangle = \frac{1}{|S|}\sum_{s \in S}h_1(s_1)h_2(s_2)h_3(s_3) = \Ex_{s \in S} h_1(s_1)h_2(s_2)h_3(s_3).
\]
This allows us to write
\begin{equation}
  \label{eq:f-g-triples}
  \Ex_{s \in S} f(s_1)f(s_2)f(s_3) - \Ex_{s \in S} g(s_1)g(s_2)g(s_3) = \sum_{j = 1}^3 \langle f-g, *_j(g, \ldots, g, f, \ldots, f) \rangle
\end{equation}
and
\begin{equation}
  \label{eq:f-g-traces}
  \sum_{x \in V} f(x) - \sum_{x \in V} g(x) = \langle f-g, \chi_V \rangle \cdot |X| \quad \text{for every $V \subseteq X$}.
\end{equation}
One of the main ideas in~\cite{CG} is to use~\eqref{eq:f-g-triples} and \eqref{eq:f-g-traces} to show that, if the inner products of two functions $f$ and $g$ with elements of some large class of functions on $X$ do not differ much from one another, then $f$ and $g$ will satisfy~\eqref{ineq:trans-triples} and~\eqref{ineq:trans-traces}. Due to various technical complications that would arise in the most straightforward approach, i.e., letting this large class of functions to contain all the convolutions $*_j$, Conlon and Gowers work with so called {\em capped convolutions} $\circ_j$, defined by $\circ_j(h_1, h_2, h_3) = \min\{*_j(h_1, h_2, h_3), 2\}$. They then define the class of {\em basic anti-uniform functions}, which are functions of the form $\circ_j(g_1, \ldots, g_j, f_{j+1}, \ldots, f_3)$, where $0 \le g_i \le 1$, $0 \le f_i \le \mu$, and $\mu$ is the associated measure of the $p$-random subset coming from some finite sequence, or of the form $\chi_V$ for some $V \in \V$ (the characteristic functions are included in order to guarantee that \eqref{ineq:trans-traces} will also hold).

After all these preparations, we can finally state the transference theorem of Conlon and Gowers~\cite{CG} in the version which is best suited for our needs.

\begin{main_ass}
Let $d,m \in \N$ and $\eta,\lambda > 0$ be given. Let $U_1, \ldots, U_m$ be independent $p$-random subsets of $X$, and let $\mu_1, \ldots, \mu_m$ be their associated measures. If $p \ge p_0$, then the following properties hold with probability $1 - n^{-C'}$, where $C'$ is a large enough constant:
  \begin{enumerate}
    \item[0.]
      $\| \mu_i \|_1 = 1 + o(1)$ for each $1 \le i \le m$.\smallskip
    \item[1.]
    If $1 \le i_1 < i_2 < i_3 \le m$, and $j \in [3]$, then
      \[
      \| *_j(\mu_{i_1}, \mu_{i_2}, \mu_{i_3}) - \circ_j(\mu_{i_1}, \mu_{i_2}, \mu_{i_3})\|_1 \le \eta.
      \]
    \item[2.]
    If $j \in \{2, 3\}$ and $1 \le i_{j+1} <  \ldots < i_3 \le m$, then
      \[
      \| *_j(1, \ldots, 1, \mu_{i_{j+1}}, \ldots, \mu_{i_3})\|_\infty \le 2.
      \]
    \item[3.]
      $\left|\langle \mu - 1, \xi \rangle\right| < \lambda$ if $\xi$ is a product of at most $d$ basic anti-uniform functions.
  \end{enumerate}
\end{main_ass}

\begin{thm}[{\cite[Theorems~4.10 and 9.3]{CG}}]
Let $\eps > 0$. Let $X$ be a finite set, let $S$ be a collection of ordered subsets of $X$ of size $3$, and let $\V$ be a collection of subsets of $X$. Then there exist constants $C, \eta, \lambda > 0$ and $d, m \in \N$ such that the following holds.  

Let $p_0$ be such that the main assumption holds for the triple $(S, p_0, \V)$ and the constants $\eta$, $\lambda$, $d$, and $m$. Let $U$ be the $p$-random subset of $X$, where $Cp_0 \le p \le 1/C$, and let $\mu$ be the associated measure of $U$. 

Then, with probability $1 - o(1)$, for every function $f \colon X \to \RR$ with $0 \le f \le \mu$, there exists a function $g \colon X \to \{0, 1\}$ such that~\eqref{ineq:trans-triples} and~\eqref{ineq:trans-traces} hold.
\end{thm}

Thus, in order to deduce that the conclusion we require, it suffices to check that the main assumption holds for our choice of $X$, $S$, $\V$ and $p_0  = C / \sqrt{n}$, for a large enough constant $C$. Indeed, Property 0 easily follows from Chernoff's inequality, as it simply says that, with probability at least $1 - n^{-C'}$, the $p$-random subset of $X$ has $(1+o(1))p|X|$ elements. Moreover, it is shown in~\cite{CG} that Property~3 is implied by Properties~0--2, together with the fact that $|\V| \le n+1 = 2^{o(p|X|)}$, so we only need to argue that our system satisfies Properties~1 and~2. 

This is done in~\cite{CG} in the case when $S$ is a homogeneous system with two degrees of freedom. In our case $S$ does have two degrees of freedom but we only know that it is `almost homogeneous', i.e., that $|S_j(x)| = (1+o(1))n$ for every $j$ and $x$. Fortunately, it is not hard (though somewhat tedious) to check that this is a sufficiently strong assumption to keep the arguments of~\cite{CG} valid; see~\cite[Lemma 7.2]{CG} and \cite[Lemma 8.4]{CG} plus the discussion below it.

Finally, let us mention briefly where the lower bound $p \ge C / \sqrt{n}$ in Theorem~\ref{thm:approx-stability} comes from. Following~\cite{CG}, for each $x \in X$ let
\[
K(x) \, = \, \big\{ y \in X \,:\, S_1(x) \cap S_3(y) \neq \emptyset \big\},
\]
and for each $y \in K(x)$, let 
\[
W(x,y) \,=\, \Ex_{s \in S_1(x) \cap S_3(y)} \mu(s_2),
\]
where $\mu$ is the associated measure of some $p$-random subset of $X$. A crucial assumption in the proof of Property~2 (see \cite[Lemma 8.4]{CG}) is that $W(x,y) \ll p|K(x)|$ for every $x,y \in X$. 

Now, since our set $S$ has two degrees of freedom, we have $|S_1(x) \cap S_3(y)| \le 1$ for every $x,y \in X$, and therefore $|K(x)| = |S_1(x)| = (1+o(1))n$. Furthermore, 
\[
W(x,y) \, \le \, \max_{s \in S_1(x) \cap S_3(y)} \mu(s_2) \, \le \, \frac{1}{p},
\]
so it is enough to require that $p^{-1} \ll p(1+o(1))n$, which is equivalent to $p \gg n^{-1/2}$.

\fi

\section{Abelian groups of Type I}\label{sec:proof-thm-main}

In this section, we shall prove Theorem~\ref{thm:main}. For the sake of clarity of the argument, from now on, we will assume not only that $q$ divides $|G|$, but also that $G$ is of type I($q$), i.e., that $q$ is the smallest prime $q'$ that divides $|G|$ and satisfies $q' \equiv 2 \pmod 3$. Since for each $q$, there are at most finitely many primes $q'$ smaller than $q$ that satisfy $q' \equiv 2 \pmod 3$, this assumption clearly does not affect the validity of our argument.

We begin with the $0$-statement, i.e., that if 
\[
\frac{\log n}{n} \, \ll \, p(n) \, \le \, c_q \sqrt{ \frac{\log n}{n} },
\]
then with high probability not all maximum-size sum-free subsets of $G_p$ are of the form $A \cap G_p$, with $A \in \SF(G)$. In fact, we shall prove that with high probability \emph{none} of them have this form. The proof uses Janson's inequality and the second moment method.

\begin{remark}
  If $\log n / n \ll p(n) \ll n^{-2/3}$, then the $0$-statement in Theorem~\ref{thm:main} becomes almost trivial. Indeed, since $G$ contains at most $n^2$ Schur triples, then with high probability the set $G_p$ itself is sum-free and (by Chernoff's inequality) $|A \cap G_p| < |G_p|$ for every $A \in \SF(G)$.
\end{remark}

\begin{proof}[Proof of the $0$-statement in Theorem~\ref{thm:main}]
We wish to prove that, for each prime $q \equiv 2 \pmod 3$, if $c_q$ is sufficiently small then the following holds with probability tending to $1$ as $n \to \infty$. Let $G$ be an Abelian group of type I($q$) with $|G| = n$, let
\[
\frac{\log n}{n} \, \ll \, p \; \le \; c_q \sqrt{ \frac{\log n}{n} },
\]
and let $G_p$ be a random $p$-subset of $G$. Then, for any maximum-size sum-free subset $B \subseteq G_p$, we have $|B| > |A \cap G_p|$ for every $A \in \SF(G)$. 

The proof will be by the second moment method. To be precise, we shall show that, given $A \in \SF(G)$, with high probability there exist at least $10\sqrt{p n \log n}$ elements $x \in G_p$, each chosen from a sum-free subset of a subgroup of $G$ disjoint from $A$, such that $(A \cap G_p) \cup \{x\}$ is sum-free. It will be easy to bound the expected number of such elements using the FKG inequality; to bound the variance we shall need to calculate more carefully, using Janson's inequality. The result then follows by Chernoff's inequality, since the size of the sets $\{A \cap G_p : A \in \SF(G)\}$ is highly concentrated.

\smallskip

To begin, observe that for any $A \in \SF(G)$, we have $|A| \ge n/3$ (by Theorem~\ref{thm:muG}) and 
\[
\Pr\left( \big| |A \cap G_p| - p|A| \big| > 4\sqrt{pn\log n}\right) \, \le \, n^{-4},
\]
by Chernoff's inequality (with $a = 4\sqrt{pn\log n}$) and by our lower bound on $p$. By Corollary~\ref{cor:SFG}, it follows that, with high probability,
  \begin{equation}
    \label{ineq:AcapGp}
    \big| |A \cap G_p| - p|A| \big| \le 4\sqrt{pn\log n} \quad \text{for every $A \in \SF(G)$}.
  \end{equation}
  
Throughout the rest of the proof, let $A \in \SF(G)$ be fixed. By Theorem~\ref{thm:SFG-structure}, there exists a subgroup $H$ of $G$ of index $q$ such that $A$ is a union of cosets of $H$. Since $H$ is not sum-free (it is a subgroup), it follows that $A \cap H = \emptyset$. Recall that $\mu(H) \ge 2/7$, by Theorem~\ref{thm:muG}, and let $E \in \SF(H)$ be an arbitrary maximum-size sum-free subset of $H$, so 
\[
|E| \, = \, \mu(H)|H| \, \ge \, \frac{2|G|}{7q}.
\]
We shall find in $E$ our elements $x$ such that $(A \cap G_p) \cup \{x\}$ is sum-free. To this end, for each $x \in E$ we define
  \begin{align*}
    C_1(x) & = \Big\{y \in A \colon x = y + y \Big\}, \\
    C_2(x) & = \left\{ \{y, z\} \in {A \choose 2} \colon x = y + z \right\}, \\
    C_3(x) & = \left\{ \{y, z\} \in {A \choose 2} \colon x = y - z \right\},
  \end{align*}
and let $C(x) = C_1(x) \cup C_2(x) \cup C_3(x)$. Note that $|C_2(x)| \le n/2$ and $|C_3(x)| \le n$ for every $x \in E$. 
  
We shall say that an element $x \in E$ is {\em safe} if no set in $C(x)$ is fully contained in $G_p$. Thus $x$ is safe if and only if the set $(A \cap G_p) \cup \{x\}$ is sum-free. We shall show below that, with high probability, $E \cap G_p$ contains more than $10\sqrt{pn\log n}$ safe elements. Since $E$ is sum-free, and $H$ is a subgroup, we have $E \pm E \subseteq H \setminus E$. Since $A \cap H = \emptyset$, it follows that the set 
\[
B \; := \; \big( A \cap G_p \big) \cup \big\{ x \in E \cap G_p \,\colon\, x \textup{ is safe} \big\}
\]
is sum-free. By~\eqref{ineq:AcapGp}, it will follow that $B$ is larger than $A' \cap G_p$ for every $A' \in \SF(G)$.

We begin by giving a lower bound on the expected number of safe elements in $E \cap G_p$. In fact, in order to simplify the calculation of the variance, below, we shall focus on a subset $E'' \subseteq E$, defined as follows. First, let $E' = \{x \in E \colon |C_1(x)| < 7q\}$ and note that, since the $C_1(x)$ are disjoint subsets of $A$, we have  
\[
\big| E \setminus E' \big| \cdot 7q \, \le \, |A| \, \le \, n,
\]
and so $|E'| \ge |E| - n/(7q) \ge n/(7q)$. Now, let $E''$ be an arbitrary subset of $E'$ of cardinality at least $|E'|/2$ such that there are no distinct $x, y \in E''$ with $x = -y$. Finally, let $S$ be the number of safe elements in $E''$.

For each $x \in E$, denote by $S_x$ the event that $x$ is safe. By the FKG inequality, we have
  \begin{equation}
    \label{PSx}
    \Pr(S_x) \ge (1-p)^{|C_1(x)|} (1-p^2)^{|C_2(x) \cup C_3(x)|}.
  \end{equation}
 Thus, since $p \le c_q \sqrt{ \frac{\log n}{n} }$, and using the bounds $|E''| \ge n/(14q)$, $|C_1(x)| \le 7q$, $|C_2(x)| \le n/2$, and $|C_3(x)| \le n$, we have
\begin{equation}
    \label{ineq:ExS}
    \Ex[S] \, = \, \sum_{x \in E''} \Pr(S_x) \, \ge \, |E''| (1-p)^{7q}(1-p^2)^{3n/2} \, \ge \, \frac{n}{20q} \cdot e^{-2p^2n} \, \ge \, \sqrt{n},
\end{equation}
  where the last inequality holds if $c_q$ is sufficiently small. 

The following bound on $\Var(S)$ will allow us to apply Chebychev's inequality.

\begin{claim}
  $\Var[S] = o(\Ex[S]^2)$.
\end{claim}

\begin{proof}
  Given distinct elements $x, y \in E$, define a graph $J(x,y)$ on the elements of $C(x) \cup C(y)$ as follows: let $B_1 \sim B_2$ if $B_1 \cap B_2 \neq \emptyset$ and $B_1 \neq B_2$. Now define
  \[
  \Delta_{x,y} = \sum_{B_1 \sim B_2} p^{|B_1 \cup B_2|},
  \]
  where the sum is taken over all edges of $J(x,y)$. For the sake of brevity, let $C_1(x,y) = C_1(x) \cup C_1(y)$ and let $C'(x,y) = C_2(x) \cup C_2(y) \cup C_3(x) \cup C_3(y)$. By Janson's inequality,
  \begin{equation}
    \label{ineq:PSxSy}
    \Pr(S_x \wedge S_y) \le (1-p)^{|C_1(x,y)|} (1-p^2)^{|C'(x,y)|} e^{\Delta_{x,y}}.
  \end{equation}
  We claim that for each $x,y \in E''$ with $x \neq y$,
 \begin{equation}
    \label{ineq:Deltaxy}
    \Delta_{x,y} \le |C_1(x,y)| \cdot 6p^2 + |C'(x,y)| \cdot 12p^3 = O\big( qp^2 + np^3 \big).
  \end{equation}
  Indeed, if $B_1 \in C_1(x,y) \cup C'(x,y)$ and $B_2 \in C'(x,y)$ are such that $B_1 \cap B_2 = \{z\}$, then $B_2 = \{z, z'\}$, where 
  \[
  z' \in \big\{ x-z, z-x, z+x, y-z, z-y, z+y \big\}.
  \]
  Thus, given $B_1 \ni z$, there are at most six sets $B_2 \in C'(x,y)$ such that $B_1 \cap B_2 = \{z\}$. It follows that for every $B_1 \in C_1(x,y)$, there are at most six sets $B_2 \in C'(x,y)$ with $B_1 \cap B_2 \neq \emptyset$ and for every $B_1 \in C'(x,y)$, there are at most twelve sets $B_2 \in C'(x,y)$ with $B_1 \cap B_2 \neq \emptyset$. The second inequality follows since $|C_1(x,y)| = |C_1(x)| + |C_1(y)| < 14q$ and $|C'(x,y)| \le 3n$. 
  
  Combining~\eqref{PSx},~\eqref{ineq:PSxSy}, and~\eqref{ineq:Deltaxy}, we obtain
  \begin{equation} \label{SxSy}
    \Pr(S_x \wedge S_y)\, \le \, \Pr(S_x) \Pr(S_y) (1-p^2)^{C^*(x,y)} e^{O( qp^2 + np^3)},
  \end{equation}
  where $C^*(x,y)  = |C'(x,y)| - |C_2(x) \cup C_3(x)| - |C_2(y) \cup C_3(y)|$. Hence, it only remains to bound $C^*(x,y)$ from below. 
  
  Recall that $E''$ does not contain any pairs $\{x,y\}$ with $x = -y$. Hence $C_3(x) \cap C_3(y) = \emptyset$, and trivially $C_2(x) \cap C_2(y) = \emptyset$. It follows by elementary manipulation that
  \begin{equation}
    \label{ineq:C2}
    |C'(x,y)| \ge |C_2(x) \cup C_3(x)| + |C_2(y) \cup C_3(y)| - |C_2(x) \cap C_3(y)| - |C_2(y) \cap C_3(x)|,
  \end{equation}
  i.e., $C^*(x,y) \ge - |C_2(x) \cap C_3(y)| - |C_2(y) \cap C_3(x)|$.
  
  If $C_2(x) \cap C_3(y) \neq \emptyset$, then there exist $a,b \in A$ such that $x = a + b$ and $y = a - b$, and thus $2a = x+y$ and $2b = x-y$. We split into two cases.
  
  \medskip
  \noindent \textbf{Case 1}: $|G|$ is odd. 
  
  \medskip
  In this case the equations $2a = x+y$ and $2b = x-y$ have at most one solution $(a,b)$, and so $|C_2(x) \cap C_3(y)| \le 1$. Hence $C^*(x,y) \ge -2$, and so, by~\eqref{SxSy},  
  \[
  \Pr(S_x \wedge S_y) \, \le \, \Pr(S_x)\Pr(S_y) (1-p^2)^{-2} e^{O(qp^2 + np^3)}.
  \]
  Thus, using the bounds $p \ll n^{-1/3}$ and $\Ex[S] \gg 1$,
  \begin{align*}
    \Var[S] & = \sum_{x,y \in E''} \Big(\Pr(S_x \wedge S_y) - \Pr(S_x)\Pr(S_y) \Big) \, \le \, \Ex[S] + \Ex[S]^2\left( (1-p^2)^{-2} e^{O(qp^2 + np^3)} - 1\right) \\
    & \le \Ex[S] + O\big( qp^2 + np^3 \big)\Ex[S]^2 \, = \, \Ex[S] + o(\Ex[S]^2) \, = \, o(\Ex[S]^2),
  \end{align*}
  as required.
  
  \medskip
  \noindent \textbf{Case 2}: $|G|$ is even.
  
  \medskip
  Since $|G|$ is even, it follows (by Theorem~\ref{thm:SFG-structure}) that $H$ is of index $2$, and so $A = G \setminus H$. Let 
  \[
  I \; := \; \big\{ h \in H \,\colon\, h + h = 0 \big\},
  \]
  note that $I$ is a subgroup of $H$, and let $i = |I|$. We claim that 
  \[
  |C_2(x) \cap C_3(y)| = 0 \quad \text{or} \quad i/2 \, \le \, |C_2(x) \cap C_3(y)| \, \le \, i.
  \]
  To see this, simply observe that if $\{a,b\} \in C_2(x) \cap C_3(y)$ and $h \in I$, then $\{a+h,b+h\} \in C_2(x) \cap C_3(y)$. Indeed, if $x = a + b$ and $y = a - b$ for some $a, b \in A$, then $x = (a+h) + (b+h)$ and $y = (a+h) - (b+h)$, and $a + h, b + h \in A$ since $A = G \setminus H$. Conversely, if $x = a+b = a'+b'$ and $y = a-b = a'-b'$ for some $a,a',b,b' \in A$, then $a' - a = b' - b \in I$. Hence, there are precisely $i$ ordered pairs $(a,b)$ with $\{a, b\} \in C_2(x) \cap C_3(y)$.

  Next, note that there are at most $n^2/(2i)$ pairs $\{x,y\} \in {E'' \choose 2}$ with $|C_2(x) \cap C_3(y)| + |C_2(y) \cap C_3(x)| > 0$. Indeed, there are at most $n^2/4$ quadruples $(a,b,x,y)$ in $A^2 \times (E'')^2$ with $x = a + b$ and $y = a - b$, and there are at least $i/2$ such quadruples for each pair $\{x,y\}$ as above. Moreover, for every $x \in E''$, either $|C_3(x)| = 0$ or $|C_3(x)| \ge n/4$, since any solution $(a,b) \in A^2$ of the equation $x = a - b$ may be shifted by an arbitrary element of $H$. Hence, for such an $x$,
  \begin{equation}
    \label{ineq:PSxSy-PSxPSy}
    \Pr(S_x \wedge S_y) - \Pr(S_x)\Pr(S_y) \, \le \, \Pr(S_x) \, \le \, (1-p^2)^{|C_3(x)|/3} \, = \, (1-p^2)^{n/12} \, \le \, e^{-p^2n/12},
  \end{equation}
  where the second inequality follows because there exists a matching in $C_3(x)$ of size $|C_3(x)|/3$. (This follows because the graph $C_3(x)$ has maximum degree $2$; the worst case is a disjoint union of triangles.)
  
  Finally, we divide once again into two cases: $i \le \sqrt{n}$ and $i > \sqrt{n}$. In the former case we have
  \[
  C^*(x,y) \, \ge \, - |C_2(x) \cap C_3(y)| - |C_2(y) \cap C_3(x)| \, \ge \, - 2\sqrt{n},
  \]
  and so, by~\eqref{SxSy},
  \[
  \Pr(S_x \wedge S_y) \,\le \,\Pr(S_x)\Pr(S_y) (1-p^2)^{-2\sqrt{n}} e^{O(qp^2 + np^3)} \, = \, (1+o(1))\Pr(S_x)\Pr(S_y),
  \]
  since $p^2\sqrt{n} + p^3n = o(1)$. Hence $\Var[S] = o(\Ex[S]^2)$, as required. 
  
  If $i > \sqrt{n}$, we partition the sum $\sum_{x,y \in E''} \big(\Pr(S_x \wedge S_y) - \Pr(S_x)\Pr(S_y) \big)$ into two parts, according to whether or not $|C_2(x) \cap C_3(y)| + |C_2(y) \cap C_3(x)| > 0$. By~\eqref{ineq:PSxSy-PSxPSy}, we obtain
  \[
  \sum_{x,y \in E''} \Big(\Pr(S_x \wedge S_y) - \Pr(S_x)\Pr(S_y) \Big) \, \le \, \Ex[S] + \Ex[S]^2 \left( e^{O(qp^2 + np^3)} - 1\right) + \frac{n^2}{2i} \cdot e^{-p^2n/12}.
  \]
  Now, recalling from~\eqref{ineq:ExS} that $\Ex[S] \ge \frac{n}{20q} e^{-2p^2n}$ and noting that $\frac{1}{\sqrt{n}} e^{-p^2n/12} \ll e^{-4p^2n}$ if $c_q$ is sufficiently small, we deduce that
  \[
  \Var[S] \, = \, o\Big( \Ex[S]^2 + n^2e^{-4p^2n} \Big) \, = \, o(\Ex[S]^2),
  \]
  as required.
\end{proof}

Since $\Var[S] = o(\Ex[S]^2)$, the number of safe elements in $E''$ satisfies 
\[
S \, \ge \, \frac{\Ex[S]}{2} \; \ge \; \frac{n}{40q} \cdot e^{-2p^2n},
\]
with high probability, by Chebyshev's inequality. Finally, note that the events $\{S_x\}_{x \in E''}$ depend only on the set $A \cap G_p$, and recall that $A \cap E'' = \emptyset$. Hence the events $\{ x \in G_p\}_{x \in E''}$ are independent of the events $\{S_x\}_{x \in E''}$, and so, by Chernoff's inequality, the number of safe elements in $E'' \cap G_p$ is at least $pS/2$ with high probability. 

Hence
\[
\left|\big\{ x \in E'' \cap G_p \colon \text{$x$ is safe} \big\} \right| \, \ge \, \frac{pn}{80q} \cdot e^{-2p^2n} \, > \, 10\sqrt{pn\log n},
\]
where the last inequality follows from the bounds $\frac{\log n}{n} \ll p < c_q \sqrt{\frac{\log n}{n}}$, provided that $c_q$ is sufficiently small and $n$ is sufficiently large. Thus, by~\eqref{ineq:AcapGp},
\[
B \; := \; \big( A \cap G_p \big) \cup \big\{ x \in E'' \cap G_p \,\colon\, x \textup{ is safe} \big\}
\]
is larger than $A' \cap G_p$ for every $A' \in \SF(G)$, and is sum-free, as required.
\end{proof}

We now turn to the $1$-statement in Theorem~\ref{thm:main}. The proof uses Theorem~\ref{thm:approx-stability-full}, together with Janson's inequality.

\begin{proof}[Proof of the $1$-statement in~Theorem~\ref{thm:main}]
Let $q$ be a prime with $q \equiv 2 \pmod 3$. We shall prove that if $C_q$ is sufficiently large, then the following holds with high probability as $n \to \infty$. Let $G$ be an Abelian group of type I($q$) with $|G| = n$, let 
\begin{equation}
  \label{eq:p-assumption}
  p \; \ge \; C_q \sqrt{ \frac{\log n}{n} },
\end{equation}
and let $G_p$ be a random $p$-subset of $G$. Then every maximum-size sum-free subset $B \subseteq G_p$ is of the form $A \cap G_p$ for some $A \in \SF(G)$. 

The proof will be roughly as follows. If a maximum-size sum-free subset $B \subseteq G_p$ is not of the form $A \cap G_p$, where $A \in \SF(G)$, then there must exist a nonempty set $S \subseteq G_p \setminus A$ and a set $T \subseteq A \cap G_p$ satisfying $|S| \ge |T|$ such that $B = S \cup (A \cap G_p) \setminus T$. We shall show that the expected number of such pairs $(S,T)$ is small when $|S| \le \eps p n$, using Janson's inequality. The case $|B \setminus A| > \eps p n$ for every $A \in \SF(G)$ is dealt with using Theorem~\ref{thm:approx-stability-full}.

\smallskip

Let $\eps > 0$ be sufficiently small and let $B$ be a maximum-size sum-free subset of $G_p$. Note that $|B| \ge |A \cap G_p|$ for any $A \in \SF(G)$ and so, by Chernoff's inequality, 
\[
|B| \; \ge \; \left( \mu(G) - \frac{\varepsilon}{40q^2 + 40q} \right) p |G|
\]
with high probability as $n \to \infty$. Hence, by Theorem~\ref{thm:approx-stability-full}, we may assume that $|B \setminus A| \le \eps p n$ for some $A \in \SF(G)$. We shall prove that in fact, with high probability, $B = A \cap G_p$. 

We shall say that a pair of sets $(S,T)$ is \emph{bad} for a set $A \in \SF(G)$ if the following conditions hold:
\begin{itemize}
\item[$(a)$] $S \subseteq G_p \setminus A$ and $T \subseteq A \cap G_p$, \\[-2ex]
\item[$(b)$] $0 < |S| = |T| \le \eps p n$,  \\[-2ex]
\item[$(c)$] $S \cup (A \cap G_p) \setminus T$ is sum-free. 
\end{itemize}
We shall prove that, for every $A \in \SF(G)$, the probability that there exists a bad pair $(S,T)$ is $o(1/n)$ as $n \to \infty$. It will follow (by Corollary~\ref{cor:SFG}) that with high probability no such pair exists for any $A \in \SF(G)$. By the above discussion, it will imply that $G_p$ is sum-free good with high probability. We remark that a bound of the form $o(1)$ on the probability of the existence of a bad pair does not suffice since the events ``$|B \setminus A| \le \eps p n$'' and ``there exists a pair $(S,T)$ which is bad for $A$'' are not independent of one another.

Fix some $A \in \SF(G)$. As in the proof of the lower bound, for every $x \in G \setminus A$, define
\[
C_1(x) = \Big\{y \in A \colon x = y + y \Big\} \quad \text{and} \quad C_2(x) = \left\{ \{y, z\} \in {A \choose 2} \colon x = y + z \right\}.
\]
We begin by proving two easy properties of the sets $C_1(x)$ and $C_2(x)$, which will be useful in what follows.

\begin{cl}
  \label{cl:claim-one}
  For every $x \in G \setminus A$, we have $\max\left\{ |C_1(x)|, |C_2(x)| \right\} \ge n/(3q)$.  
\end{cl}
\begin{proof}
  By Theorem~\ref{thm:SFG-structure}, there exists a subgroup $H$ of $G$ of index $q$ such that $A$ is a union of cosets of $H$ and $A \cup (A+A) = G$. It follows that $x = y + z$ for some $y, z \in A$, and that $y+h, z-h \in A$ for every $h \in H$. Thus $\{y + h,z - h\} \in C_1(x) \cup C_2(x)$ for every $h \in H$, and hence $|C_1(x)| + 2|C_2(x)| \ge |H| = n/q$.
\end{proof}

Let $A^* = \{x \in G \setminus A \colon |C_1(x)| \ge n/(3q)\}$. We now show that we can immediately disregard bad pairs that contain an element of $A^*$.
  
\begin{cl}
  \label{cl:claim-two}
  With probability $1 - o(1/n)$, there is no bad pair $(S,T)$ with $S \cap A^* \neq \emptyset$.
\end{cl}
\begin{proof}
  Let $x \in A^*$, so $x \in G \setminus A$ and $|C_1(x)| \ge n/(3q)$. If $(S,T)$ is a bad pair for $A$ with $x \in S$, then $C_1(x) \cap G_p \subseteq T$, since $S \cup (A \cap G_p) \setminus T$ is sum-free, and $|T| \le \eps p n$. But by Chernoff's inequality, we have
  \[
  \left| C_1(x) \cap G_p \right| \, \ge \, \frac{pn}{6q} \, > \, \eps p n \, \ge \, |T|
  \]
  with probability at least $1 - \exp\big( -\frac{pn}{24q} \big) \ge 1 - n^{-3}$. By the union bound, with probability $1 - o(1/n)$ this holds for every $x \in A^*$, and hence with probability $1 - o(1/n)$ there is no bad pair $(S,T)$ with $S \cap A^* \neq \emptyset$.
\end{proof}

We now arrive at an important definition. Given a set $S \subseteq G \setminus (A \cup A^*)$ of size $k$, let $\G_S$ denote the graph with vertex set $A$ and edge set $\bigcup_{x \in S}C_2(x)$. The key observation about $\G_S$ is the following: If $(S,T)$ is a bad pair for $A$, then $A \cap G_p \setminus T$ is an independent set in $\G_S$. This follows because the endpoints of each edge of $\G_S$ sum to an element of $S$.

By Claim~\ref{cl:claim-one} and the definition of $A^*$, we have that $|C_2(x)| \ge n/(3q)$ for every $x \in S$. Since, by definition, $C_2(x) \cap C_2(x') = \emptyset$ if $x \neq x'$, we have $e(\G_S) \ge kn/(3q)$. Moreover, each $C_2(x)$ is a matching (i.e., no two edges share an endpoint) and hence $\Delta(\G_S) \le |S| = k$. 

Given a subset $T \subseteq A$ of size $k$, let $\G_{S,T} = \G_S[A \setminus T]$. Note that $e(\G_{S,T}) \ge kn/(3q) - k^2$ and $\Delta(\G_{S,T}) \le k$ for every such $T$. Also, if $(S,T)$ is a bad pair for $A$, then $e(\G_{S,T}[G_p]) = 0$. Crucially, since the sets $S$, $T$, and $A \setminus T$ are pairwise disjoint, the events $S \subseteq G_p$, $T \subseteq G_p$, and $e(\G_{S,T}[G_p]) = 0$ are independent. Hence, the expected number of bad pairs for $A$ is at most
  \begin{equation}
    \label{ineq:P-bad-pair}
    \sum_{S,T} \, \Pr\big( S  \subseteq G_p \big) \, \Pr \big( T \subseteq G_p \big) \, \Pr\big( e(\G_{S,T}[G_p]) = 0 \big),
  \end{equation}
  where the summation ranges over all $S \subseteq G \setminus (A \cup A^*)$ and $T \subseteq A$ with $1 \le |S| = |T| \le \eps p n$.

  Fix a $k$ with $1 \le k \le \eps p n$ and let $S \subseteq G \setminus (A \cup A^*)$ and $T \subseteq A$ be sets of size $k$. Let $\mu = \Ex[e(\G_{S,T}[G_p])]$ and observe that
  \begin{equation}
    \label{ineq:mu}
    \mu \, = \,  p^2e\big( \G_{S,T} \big) \, \ge \, p^2 \left( \frac{kn}{3q} - k^2 \right) \, \ge \, \frac{p^2kn}{6q}.
  \end{equation}
since $k \le \varepsilon n$ and $\eps$ is sufficiently small. Furthermore, let $\Delta = \sum_{B_1 \sim B_2} p^{|B_1 \cup B_2|}$, where the summation ranges over all $B_1, B_2 \in E(\G_{S,T})$ such that $B_1 \neq B_2$ and $B_1 \cap B_2 \neq \emptyset$ (in other words, $B_1$ and $B_2$ are two edges of $\G_{S,T}$ that share an endpoint), and note that
  \begin{equation}
    \label{ineq:Delta}
    \Delta \, \le \, n {\Delta(\G_{S,T}) \choose 2} p^3 \, \le \, p^3k^2n.
  \end{equation}
  By Janson's inequality, if $\Delta \le \mu$, then~\eqref{ineq:mu} and~\eqref{eq:p-assumption} imply that
  \begin{equation*}
    \label{ineq:Janson-one}
    \Pr\big( e(\G_{S,T}[G_p]) = 0 \big) \, \le \, e^{-\mu/2} \, \le \, \left( e^{-p^2n/(12q)} \right)^k \, \le \, n^{-5k}
  \end{equation*}
  and if $\Delta \ge \mu$, then~\eqref{ineq:Delta} implies that
  \begin{equation*}
    \label{ineq:Janson-two}
    \Pr\big( e(\G_{S,T}[G_p]) = 0 \big) \, \le \, e^{-\mu^2/(2\Delta)} \, \le \, e^{-pn/(72q^2)}.
  \end{equation*}
 Hence, using~\eqref{ineq:P-bad-pair}, for those $k$ such that $\Delta \le \mu$, the expected number of bad pairs $(S,T)$ with $|S| = |T| = k$ is at most
  \[
  {n \choose k}^2p^{2k} n^{-5k} \, \le \, n^{-3k} \, \le \, n^{-3}.
  \]
 Also, recalling that $k \le \eps p n$, for those $k$ such that $\Delta > \mu$, the expected number of bad pairs $(S,T)$ with $|S| = |T| = k$ is at most
  \[
  {n \choose k}^2p^{2k} e^{-\frac{pn}{72q^2}} \, \le \, \left(\frac{epn}{k}\right)^{2k} e^{-\frac{pn}{72q^2}} \, \le \, \left[\left(\frac{e}{\eps}\right)^{2\eps} e^{-\frac{1}{72q^2}} \right]^{pn} \, \le \, e^{-\sqrt{n}}
  \]
  whenever $\eps$ is sufficiently small and $n$ is sufficiently large (depending only on $q$). It follows that $\eps p n \cdot \max\{ n^{-3}, e^{-\sqrt{n}} \}$ is an upper bound on the probability that there  exists a bad pair $(S,T)$ for $A$. Since $|\SF(G)| \le n$, the expected number of pairs $(S,T)$ which are bad for some $A \in \SF(G)$ tends to $0$ as $n \to \infty$. This completes the proof.
\end{proof}

\section{Lower bounds for other Abelian groups}\label{sec:disc-thm-main}

In this section we shall prove the following three propositions, which show that the threshold for some groups can be much larger than that determined in Theorems~\ref{thm:sharp} and~\ref{thm:main}. The first of these results shows that for certain groups of Type II, the threshold is at least $(n\log n)^{-1/3}$. Recall that $G_p$ denotes the $p$-random subset of $G$.

\begin{prop}
  \label{c_ex_2}
  Let $q$ be a prime with $q \equiv 1 \pmod 3$, let $n = 3q$, and let $G = \ZZ_n$. If 
  \[
  \frac{\log n}{n} \; \ll \; p(n) \; \ll \; \left( \frac{1}{n\log n} \right)^{1/3},
  \]
  then, with high probability as $n \to \infty$, there exists a sum-free subset of $G_p$ which is larger than $A \cap G_p$ for every $A \in \SF(G)$.
\end{prop}

The second result gives the same bounds for groups of Type I of prime order.

\begin{prop}
  \label{c_ex_1}
  Let $q$ be a prime with $q \equiv 2 \pmod 3$, let $n = q$, and let $G = \ZZ_n$. If 
  \[
  \frac{\log n}{n} \; \ll \; p(n) \; \ll \; \left( \frac{1}{n\log n} \right)^{1/3},
  \]
  then, with high probability as $n \to \infty$, there exists a sum-free subset of $G_p$ which is larger than $A \cap G_p$ for every $A \in \SF(G)$.
\end{prop}

The third proposition shows that for the hypercube $\{0,1\}^k$ on $2n$ vertices, the threshold is different from the threshold for $\ZZ_{2n}$.

\begin{prop}\label{2^k}
  Let $C < 1/2$ and let $k, n \in \N$ satisfy $n = 2^{k-1}$, and let $G = \ZZ_2^k$. If 
  \[
  \frac{\log n}{n} \; \ll \; p \; \le \; \sqrt{\frac{C \log n}{n} },
  \]
  then, with high probability as $n \to \infty$, there exists a sum-free subset of $G_p$ which is larger than $A \cap G_p$ for every $A \in \SF(G)$.
\end{prop}

The proofs of Propositions~\ref{c_ex_2} and~\ref{c_ex_1} are almost identical and are based on the following general statement providing a lower bound for the size of a largest sum-free subset of the $p$-random subset of $\ZZ_n$.

\begin{lemma}
  \label{thm:lower-bound-Zn}
  Let $G = \ZZ_n$ and let $m = \min\{n,p^{-2}\}/100$. If 
  \[
  \frac{\log n}{n} \; \ll \; p(n) \; \ll \; 1,
  \]
  then, with high probability as $n \to \infty$, the largest sum-free subset of $G_p$ has at least 
  \[
  \frac{pn}{3} \, + \, \frac{pm}{4} \, - \, 4\sqrt{pn\log n}
  \]
  elements.
\end{lemma}

\begin{proof}
  The idea is as follows: First, we construct a sum-free set $A \subseteq G$ of size about $n/3 - 2m$; then we observe that $A \cap G_p$ typically has at least $pn/3 - 2pm - 4\sqrt{pn\log n}$ elements; finally, we show that, with high probability, we can add $9pm/4$ elements to  $A \cap G_p$ while still remaining sum-free.  

  We begin by letting $A = \{\ell, \ldots, r\}$, where $\ell = \lceil n/3 \rceil + \lceil 4m \rceil + 1$ and $r = \lfloor 2n/3 \rfloor + \lfloor 2m \rfloor$. Observe that $A$ is sum-free, since $2\ell > r$ and $2r - n < \ell$, and that
  \[
  \frac{n}{4} \, \le \, \frac{n}{3} - 2m - 3 \, \le \, |A| \, \le \, \frac{n}{3} - 2m.
  \]
  By Chernoff's inequality (applied with $a = 3\sqrt{pn\log n}$) and our lower bound on $p$, we have
  \begin{equation}
    \label{ineq:AcapGp2}
    \left| \big| A \cap G_p \big| - \left( \frac{pn}{3} - 2pm \right) \right| \, \le \, 4\sqrt{pn\log n}
  \end{equation}
  with high probability.

  Next, let $A' = \{\ell', \ldots, \ell-1\}$ and let $A'' = \{r', \ldots, r\} \subseteq A$, where $\ell' = \lceil n/3 \rceil + \lceil m \rceil + 1$ and $r' = \lfloor 2n/3 \rfloor - \lceil m \rceil$. Since $A$ is sum-free, $2\ell' > r$ and $(r' - 1) + r - n < \ell'$, it follows that every Schur triple $(x,y,z)$ in $A \cup A'$ satisfies $x, y \in A''$ and $z \in A'$.
  
  For every $x \in A'$, let
  \[
  C_1(x) = \Big\{ y \in A'' \colon x = y + y \Big\}, \quad C_2(x) = \Big\{ \{y, z\} \in {{A''} \choose 2} \colon x = y + z \Big\},
  \]
  and $C(x) = C_1(x) \cup C_2(x)$. Moreover, note that $|C_1(x)| \le 1$ and $|C_2(x)| \le |A''|/2 \le 2m$ for every $x \in A'$. Call an $x \in A'$ {\em safe} if no $B \in C(x)$ is fully contained in $G_p$. By the above observation about the Schur triples in $A \cup A'$, the set 
  \[
  \big( A \cap G_p \big) \cup \big\{ x \in A' \colon x \text{ is safe} \big\}
  \]
  is sum-free.

  In the remainder of the proof, we will show that a.a.s.~$A' \cap G_p$ contains at least $9pm/4$ safe elements. Together with~\eqref{ineq:AcapGp2}, this will imply the existence of a sum-free subset of $G_p$ with at least $pn/3 + pm/4 - 4\sqrt{pn \log n}$ elements.

  For each $x \in A'$, denote by $S_x$ the event that $x$ is safe, and let $S$ be the number of safe elements in $A'$. As in the proof of the lower bound in Theorem~\ref{thm:main}, by the FKG inequality we have
  \[
  \Ex[S] \, = \, \sum_{x \in A'} \Pr(S_x) \, \ge \, \sum_{x \in A'} (1-p)^{|C_1(x)|} (1-p^2)^{|C_2(x)|} \, \ge \, |A'| e^{-3p^2m} \, \ge \, \frac{11m}{4}.
  \]
Here we used the bounds $|C_1(x)| \le 1$, $|C_2(x)| \le 2m$, $|A'| \ge 3m - 2$ and $p^2m \le 1/100$. Similarly, using Janson's inequality, we obtain
  \begin{align*}
    \Var[S] & = \sum_{x,y \in A'} \Big( \Pr(S_x \wedge S_y) - \Pr(S_x)\Pr(S_y) \Big) \, \le \, \Ex[S] + \Ex[S]^2\left( e^{O(p^2 + p^3m)} - 1\right)\\
    & \le \Ex[S] + O\big( p^2 + p^3m \big) \cdot \Ex[S]^2 \, = \, o\big( \Ex[S]^2 \big).
  \end{align*}
To see this, observe that~\eqref{ineq:PSxSy},~\eqref{ineq:Deltaxy}, and~\eqref{SxSy} still hold (with $n$  replaced by $2m$) and that $C^*(x,y) = 0$. The last inequality follows from the fact that $p^2 + p^3m \ll 1$. 

By Chebyshev's inequality,
  \[
 \Pr\left( \big| S - \Ex[S] \big| > \frac{\Ex[S]}{11} \right) \, \le \, \frac{121\Var[S]}{\Ex[S]^2} \, = \, o(1),
  \]
and it follows that, with high probability, the number of safe elements in $A'$ satisfies 
  \[
  S \, \ge \, \frac{10\Ex[S]}{11} \, \ge \, \frac{5m}{2}.
  \]
Finally, note that for every $x \in A'$, the event $S_x$ is independent of the events $\{ y \in G_p \}_{y \in A'}$. Hence, by Chernoff's inequality, with high probability the number of safe elements in $A' \cap G_p$ satisfies
  \[
  \left| \big\{ x \in A' \cap G_p \colon \text{$x$ is safe} \big\} \right| \, \ge \, \frac{9p S}{10} \, \ge \, \frac{9pm}{4},
  \]
as required.
\end{proof}

Propositions~\ref{c_ex_2} and~\ref{c_ex_1} both follow easily from Lemma~\ref{thm:lower-bound-Zn}; since the proofs are almost identical, we shall prove only the former and leave the details of the latter to the reader.

\begin{proof}[Proof of Proposition~\ref{c_ex_2}]
  Fix an $A \in \SF(G)$ and recall that $|A| = n/3$ by Theorem~\ref{thm:muG}. (Under the assumptions of Proposition~\ref{c_ex_1}, we would have that $|A| = \lceil n/3 \rceil$, again by Theorem~\ref{thm:muG}). Thus, by Chernoff's inequality (with $a = 4\sqrt{pn\log n}$) and our lower bound on $p$, we have
  \[
  \Pr\left( \left| |A \cap G_p| - \frac{pn}{3} \right| > 4\sqrt{pn\log n} \right) \, \le \, \frac{1}{n^4}.
  \]
  Hence, by Corollary~\ref{cor:SFZ3q} (Corollary~\ref{cor:SFG} in the proof of Proposition~\ref{c_ex_1}), with high probability,
  \begin{equation}
    \label{ineq:AcapGp-lower}
    |A \cap G_p| \, \le \, \frac{pn}{3} + 4\sqrt{pn\log n} \quad \text{for every $A \in \SF(G)$}.
  \end{equation}
  Now, by Theorem~\ref{thm:lower-bound-Zn}, with high probability $G_p$ contains a sum-free subset with at least 
  \[
  \frac{pn}{3} \, + \, \frac{pm}{4} \, - \, 4\sqrt{pn\log n}
  \]
  elements, where $m = \min\{n, p^{-2}\}/100$. 

  Finally, observe that $pm \gg \sqrt{pn\log n}$, since if $m = n/100$, then this is equivalent to $pn \gg \log n$ and if $m = 1/(100p^2)$, then it is equivalent to $p^3 n \log n \ll 1$. Hence, by~\eqref{ineq:AcapGp-lower}, there exists a sum-free subset of $G_p$ that is larger than $A \cap G_p$ for every $A \in \SF(G)$.
\end{proof}

We shall now prove Proposition~\ref{2^k}. Note that for the hypercube, the conditions $x = y + z$ and $x = y - z$ are the same and so we have fewer restrictions for a vertex to be safe. This allows us to show that the threshold is different than in the case $G = \ZZ_{2n}$. 

\begin{proof}[Proof of Proposition~\ref{2^k}]
  With each non-zero element $a \in G = \ZZ_2^k$, we can associate the subgroup $E(a)$ of elements that are orthogonal to $a$ (when viewing $\ZZ_2^k$ as a linear space over the $2$-element field); note that $E(a)$ has index $2$ and let $O(a) = G \setminus E(a)$ to be its non-zero coset. Note that $n = |E(a)| = |O(a)|$, that $O(a)$ is sum-free, and that, by Theorem~\ref{thm:SFG-structure}, every element of $\SF(G)$ is of the form $O(a)$ for some non-zero $a \in G$. By Chernoff's inequality, for every $a \in G \setminus \{0\}$,
  \[
  \Pr\left( \big| |O(a) \cap G_p| - pn \big| > 4\sqrt{pn\log n}\right) \, \le \, \frac{1}{n^2},
  \]
  and hence, with high probability,
  \[
  \big| |O(a) \cap G_p| - pn \big| \le 4\sqrt{pn\log n} \quad \text{for every $a \in G \setminus \{0\}$}.
  \]
  Let $\one \in G$ be the all-ones vector. To simplify the notation, let $E = E(\one)$ and $O = O(\one)$. Since $E$ is isomorphic to $\ZZ_2^{k-1}$, $E$ contains a sum-free subset of size $2^{k-2}$. Choose one such subset and denote it $E'$. Note that $0 \not\in E'$. 

  For each $x \in E'$, let
  \[
  C(x) = \left\{ \{y, z\} \in {O \choose 2} \colon x = y + z \right\},
  \]
  and note that $|C(x)| = n/2$. Call an $x \in E'$ {\em safe} if no $B \in C(x)$ is fully contained in $G_p$, and for every $x \in E'$, denote by $S_x$ the event that $x$ is safe. Let $\eps = 1/4 - C/2 > 0$, let
  \[
  \mu \, := \, \frac{np^2}{2} \, \le \, \frac{C\log n}{2} \, = \, \left( \frac{1}{4} - \eps \right) \log n,
  \]
  and observe that, by the FKG inequality,
  \[
  \Pr(S_x) \, \ge \, (1-p^2)^{n/2} \, \ge \, \exp\Big( - \mu - O\big( np^4 \big) \Big) \, \ge \, \frac{n^{-1/4+\eps}}{2},  
  \]
  since $np^4 \ll 1$. Let $S$ be the number of safe elements in $E'$ and note that
  \[
  \Ex[S] \, = \, \sum_{x \in E'}\Pr(S_x) \, \ge \, \frac{n^{3/4 + \eps}}{4}.
  \]
  Since for two distinct $x, y \in E'$, the sets $C(x)$ and $C(y)$ are disjoint, it follows from Janson's inequality (as in the proof of Theorem~\ref{thm:main}) that
  \[
  \Pr(S_x \wedge S_y) \le (1-p^2)^{|C(x) \cup C(y)|} e^{np^3} = (1+o(1))\Pr(S_x)\Pr(S_y)
  \]
  and hence $\Var[S] = o(\Ex[S]^2)$. By the inequalities of Chebyshev and Chernoff (as in the proof of Lemma~\ref{thm:lower-bound-Zn}), with high probability the set 
  \[
  \big( O \cap G_p \big) \cup \big\{ x \in E' \cap G_p \colon \text{$x$ is safe} \big\}
  \]
  is sum-free and larger than $O(a) \cap G_p$ for every non-zero $a \in G$.
\end{proof}

\section{The group $\ZZ_{2n}$}\label{sec:proof-thm-Z2n}

Let $E_{2n}$ and $O_{2n}$ denote the sets of even and odd elements of $\ZZ_{2n}$, respectively, and recall that $O_{2n}$ is the unique maximum-size sum-free subset of $\ZZ_{2n}$. Let $G = \ZZ_{2n}$ and recall that $G_p$ denotes the $p$-random subset of $G$. We shall prove Theorem~\ref{thm:sharp}, which gives a sharp threshold for the property that $\SF(G_p) = \{ G_p \cap O_{2n} \}$.

For each $x \in E_{2n} \cap [n-1]$, let
\begin{align*}
  C_1(x) & = \Big\{ y \in O_{2n} \colon x = y + y \Big\}, \\
  C_2(x) & = \left\{ \{y, z\} \in {O_{2n} \choose 2} \colon x = y + z \right\}, \\
  C_3(x) & = \left\{ \{y, z\} \in {O_{2n} \choose 2} \colon x = y - z \right\},
\end{align*}
and let $C(x) = C_1(x) \cup C_2(x) \cup C_3(x)$. Call such an element $x$ {\em safe} if no $B \in C(x)$ is fully contained in $G_p$ and observe that $x$ is safe if and only if $(G_p \cap O_{2n}) \cup \{x\}$ is sum-free.

The following heuristic argument explains where the constant $1/3$ in the statement of Theorem~\ref{thm:sharp} comes from. To establish the $0$-statement, it is enough to prove that, with high probability, $G_p$ contains some safe element. Since $|C_1(x)| \le 2$ and $|C_2(x) \cup C_3(x)| \approx 3n/2$ for all $x \in E_{2n} \cap [n-1]$, we have that
\begin{equation}
  \label{eq:Ex-safe-elements}
  \Ex\big[\text{\#safe elements in $G_p$}\big] \approx p \cdot (1-p^2)^{3n/2} \cdot |E_{2n} \cap [n-1]| \approx pe^{-3p^2n/2}n/2.
\end{equation}
Now, note that the right-hand side of~\eqref{eq:Ex-safe-elements} grows with $n$ precisely when $p^2n \le \frac{1}{3}\log n$, so, at least in expectation, the sharp threshold for the property of containing a safe element is at $p = \sqrt{\log n / (3n)}$.  To establish the $1$-statement, one needs to prove that, with high probability, there are no $S \subseteq G_p \cap E_{2n}$ and $T \subseteq G_p \cap O_{2n}$ with $|T| \le |S|$ and $|S| \ge 1$ such that $(G_p \setminus T) \cup S$ is sum-free. Theorem~\ref{thm:approx-stability-full} rules out the possibility that such sets exists with $|S| = \Omega(pn)$. Otherwise, it turns out that the `worst' case is already when $|S|=1$ and $|T| = 0$; showing this is the most difficult and technical part of the proof. In other words, the threshold for the property of being sum-free good is the same as the threshold for the property of (not) containing a safe element.

We now turn to the rigorous proof of Theorem~\ref{thm:main} and begin by proving the $0$-statement. Even though we compute the precise value of the threshold, the proof is somewhat simpler than in the general case.

\begin{proof}[Proof of the $0$-statement in Theorem~\ref{thm:sharp}]
  Assume that 
  \[
  \frac{\log n}{n} \, \ll \, p \; \le \; \sqrt{ \frac{\log n}{3n} }
  \]
  and let $G_p$ be the $p$-random subset of $G = \ZZ_{2n}$. We shall prove that, with high probability, $G_p \cap E_{2n}$ contains a safe element and hence there exists a sum-free subset $B \subseteq G_p$ that is larger than $G_p \cap O_{2n}$. 

  Indeed, for each $x \in E_{2n} \cap [n-1]$, denote by $S_{x}$ the event that $x$ is safe and let $S$ denote the number of safe elements in $E_{2n} \cap [n-1]$. Note that $|C_1(x)| \le 2$, $\frac{n}{2} - 1 \le |C_2(x)| \le \frac{n}{2}$, and $|C_3(x)| = n$ for every $x \in E_{2n} \cap [n-1]$, where we used the fact that $x \neq -x$ for every such $x$. By the FKG inequality, 
\[
    \Pr(S_{x}) \, \ge \, (1-p)^{|C_1(x)|}(1-p^2)^{|C_2(x) \cup C_3(x)|} \, \ge \, \exp\left( - \frac{3p^2n}{2} + O\big( p + p^4n \big) \right),
\]
  and so, since $3p^2 n \le \log n$,
\begin{equation}
    \label{eq:ExS}
  \Ex[S] \, \ge \, \sum_{x \in E_{2n} \cap [n-1]}  \Pr(S_{x}) \, \ge \, \frac{n}{3} \cdot e^{-3p^2n/2} \, \ge \, \frac{1}{3} \sqrt{n}.
\end{equation}
  The calculation of the variance is similar to those in the proofs of Theorem~\ref{thm:main} and Lemma~\ref{thm:lower-bound-Zn} and we omit the details, noting only that $C_2(x) \cap C_2(y) = C_3(x) \cap C_3(y) = \emptyset$ and $|C_2(x) \cap C_3(y)| \le 4$ if $x \neq -y$. Using the fact that $E_{2n} \cap [n-1]$ and -$(E_{2n} \cap [n-1])$ are disjoint, we obtain
  \begin{align*}
    \Var[S] & = \sum_{x,y \in E_{2n} \cap [n-1]} \Big(\Pr(S_x \wedge S_y) - \Pr(S_x)\Pr(S_y) \Big) \, \le \, \Ex[S] + \Ex[S]^2 \left( e^{O(p^2 + np^3)} - 1\right) \\
    & \le \Ex[S] + O\big(p^2 + np^3 \big) \Ex[S]^2 \, = \, o\big( \Ex[S]^2 \big),
  \end{align*}
  since $\Ex[S] \gg 1$, by~\eqref{eq:ExS}, and $p^3n = o(1)$. Hence, by Chebyshev's inequality, with high probability the number of safe elements satisfies 
  \[
  S \, \ge \, \frac{\Ex[S]}{2} \, \ge \, \frac{n}{6} \cdot e^{-3p^2n/2}.
  \]

  Finally, for each $x \in E_{2n}$, the event $x \in G_p$ is independent of all the events $\{S_x\}_{x \in E_{2n} \cap [n-1]}$, so the probability that no safe element belongs to $G_p$ (given $|S| \ge \Ex[S]/2$) is at most 
  \[
  (1-p)^{\Ex[S]/2} \, \le \, \exp\left( -\frac{pn}{6} \cdot e^{-3p^2n/2} \right) \, \le \, \exp\left(-\frac{\sqrt{\log n}}{20}\right) \, \le \, o(1),
  \]
  as required.
\end{proof}

In order to obtain the $1$-statement, we shall have to work much harder. To show that all sum-free sets $B$ with at least $\eps p n$ even numbers are smaller than $G_p \cap O_{2n}$, we shall use the result of Conlon and Gowers (Theorem~\ref{thm:approx-stability-full}), as in the proof of Theorem~\ref{thm:main}. When $|B \cap E_{2n}| = o(pn)$, however, we shall need to study carefully the structure of the graph $\G_S$ for each $S \subseteq E_{2n}$, where $V(\G_S) = O_{2n}$ and $E(\G_S)$ consists of all pairs $\{a, b\}$ such that either $a + b \in S$, or $a - b \in S$.

\begin{proof}[Proof of the $1$-statement in Theorem~\ref{thm:sharp}]
  Let $\delta > 0$, let $C = C(n) \ge \ds\frac{1}{3} + \delta$, let
  \[
  p \, = \, p(n) \, = \, \sqrt{ \frac{C \log n}{n} },
  \]
  and let $G_p$ be the $p$-random subset of $G = \ZZ_{2n}$. We shall prove that, with high probability, $G_p \cap O_{2n}$ is the unique maximum-size sum-free subset of $G_p$. 

  Let $\eps > 0$ be sufficiently small and let $B$ be a maximum-size sum-free subset of $G_p$. Note that $|B| \ge |G_p \cap O_{2n}|$ and so, by Chernoff's inequality, 
  \[
  |B| \; \ge \; \left( \frac{1}{2} - \eps^2 \right) p |G|
  \]
  with high probability as $n \to \infty$. Hence, by Theorem~\ref{thm:approx-stability-full}, we have $|B \setminus O_{2n}| \le \eps p n$ with high probability. 

  Let $S = B \cap E_{2n}$ and suppose that $|S| = k$ for some positive $k \le \eps p n$. Since $B$ is at least as large as $G_p \cap O_{2n}$, there must exist a set $T \subseteq G_p \cap O_{2n}$, with $|T| \le k$, such that 
  \[
  B \, = \, S \cup \big(G_p \cap O_{2n} \big) \setminus T.
  \]
  We shall bound the expected number of such pairs $(S,T)$, with $|S| = k$ and $T$ minimal.

  For each set $S \subseteq E_{2n}$, define $\G_S$ to be the graph with vertex set $O_{2n}$ whose edges are all pairs $\{a, b\} \in {O_{2n} \choose 2}$, such that either $a + b \in S$ or $a - b \in S$. Note that we ignore loops\footnote{Hence we prove a slightly stronger result -- even if we allow Schur triples of the form $(a,a,2a)$ to be contained in a sum-free set, the odd numbers are still the best.}. We say that a pair of sets $(S, T)$ is {\em good} if the following conditions hold:
  \begin{itemize}
  \item[$(i)$] $S \subseteq G_p \cap E_{2n}$, with $|S| = k$,\\[-2ex]
  \item[$(ii)$] $T \subseteq G_p \cap O_{2n}$, with $|T| \le k$,\\[-2ex]
  \item[$(iii)$] $G_p \setminus T$ is an independent set in $\G_S$, \\[-2ex]
  \item[$(iv)$] $G_p \setminus T'$ is not independent for every $T' \subsetneq T$.
  \end{itemize}
  It is a simple (but key) observation that if $B = S \cup (G_p \cap O_{2n}) \setminus T$ is a maximum-size sum-free subset of $G_p$ with $|T| \le |S| = k$, then $(S, T')$ is a good pair for some $T' \subseteq T$. Indeed, every edge of $\G_S[G_p \cap O_{2n}]$ must have an endpoint in $T$ since $S \cup (G_p \cap O_{2n}) \setminus T$ is sum-free. Now take $T' \subseteq T$ to be minimal such that this holds. 

  Let $m$ denote the number of pairs $\{x,-x\} \subseteq S$, where $x \in E_{2n} \cap [n-1]$, and let $\one_S(x)$ denote the indicator function of the event $x \in S$. 

  \begin{cl}
    \label{cl:claim-one-Z2n}
    $e(\G_S) \ge \left( \ds\frac{3k-\one_S(n)}{2} - m \right) n - O\big( k^2 \big)$ and $\Delta(\G_S) \le 3k$.
  \end{cl}
  \begin{proof}
    Each $x \in S$ (other than $x = n$, which contributes $n/2$) contributes $n$ edges $\{a, b\}$ with $a - b = x$ to $\G_S$, with two edges incident to each vertex of $O_{2n}$. The edge sets corresponding to different members of $S$ are disjoint, except for those corresponding to $x$ and $-x$, which are identical. Thus the pairs $\{a, b\}$ with $a - b \in S$ contribute at least $\big( k - m - \frac{1}{2} \one_S(n) \big)n$ edges to $\G_S$, and there are at most $2(k - m) \le 2k$ such edges incident to each vertex. 

    Now consider the pairs $\{a, b\}$ with $a + b \in S$. Each $x \in S$ contributes at least $(n-2)/2$ such edges (since there are at most two loops, at $x/2$ and $n + x/2$) and these sets of edges are disjoint for different members of $S$. Moreover, each vertex is incident to at most $k$ such edges, and so $\Delta(\G_S) \le 3k$.

    Finally, if for some $\{a, b\}$, both $a + b$ and $a - b$ are in $S$, then $\{a, b\} = \{x+y, x-y\}$ or $\{a,b\} = \{n+x+y, n+x-y\}$, where $2x, 2y \in S$. There are at most $k^2$ such pairs $\{2x,2y\}$, and so the number of pairs $\{a, b\}$ with $a + b \in S$ and $a - b \not\in S$ is at least $nk/2 - 3k^2$. Hence the total number of edges is at least
    \[
    \left( k - m - \frac{\one_S(n)}{2} \right) n \,+\, \frac{nk}{2} \, - \, 3k^2,
    \]
    as required.
  \end{proof}

  The following idea is key: 
  
  \begin{cl}
    \label{cl:claim-two-Z2n}
    If a pair $(S, T)$ is good, then there exists a subset $U \subseteq (G_p \cap O_{2n}) \setminus T$ with $|U| \le |T|$, such that $T \subseteq N_{\G_S}(U)$ and in $\G_S$ there is matching of size $|U|$ from $U$ to $T$.
  \end{cl}
  
  \begin{proof}
    To see this, we simply take a maximal matching $M$ from $T$ to $G_p \setminus T$ in $\G_S$ and let $U$ be the set of vertices in $G_p \setminus T$ that are incident to $M$. By construction, $|U| = |M| \le |T|$ and $M$ is a matching of size $|U|$ from $U$ to $T$. 
  
    It remains to prove that $T \subseteq N_{\G_S}(U)$. Suppose not, i.e., assume that there is a vertex $a \in T \setminus N_{\G_S}(U)$. Since $G_p \setminus T$ is an independent set in $\G_S$ and $T$ is minimal, it follows that $a$ has a neighbor $b \in G_p \setminus T$. But $a \not\in N_{\G_S}(U)$, so $b \not\in U$, and thus $M \cup \{a,b\}$ is a matching from $T$ to $G_p \setminus T$ which is larger than $M$. This contradicts the choice of $M$. It follows that $T \subseteq N_{\G_S}(U)$, as claimed.
  \end{proof}

  The plan for the rest of the proof is the following: Say that a triple $(S,T,U)$ is good if $(S, T)$ is good, $U \subseteq (G_p \cap O_{2n}) \setminus T$ and $T \subseteq N_{\G_S}(U)$. Let $Z' = Z'(k,\ell,m,j)$ denote the number of good triples $(S, T, U)$ with $|S| = k$, $|T| = \ell$, $|U| = j$, and with $m$ pairs $\{x,-x\}$ in $S$, and let
  \[
  Z \; := \; \sum_{k=1}^{\eps p n}\sum_{m = 0}^{k/2}\sum_{\ell = 0}^k\sum_{j=0}^{\ell} Z'(k,\ell,m,j).
  \]
  By Claim~\ref{cl:claim-two-Z2n} and the discussion above, if there exists a maximum-size sum-free set $B \ne G_p \cap O_{2n}$, then $Z \ge 1$, and hence
  \begin{equation}\label{PatmostZ}
    \Pr\Big( \SF(G_p) \ne \{G_p \cap O_{2n}\} \Big) \, \le \, \Ex\big[ Z \big] \, = \, \sum_{k,\ell,m,j} \Ex\big[ Z'(k,\ell,m,j) \big].
  \end{equation}
  It thus will suffice to bound $\Ex[ Z'(k,\ell,m,j) ]$ for each $k$, $\ell$, $m$, and $j$. We shall prove that $\Ex[ Z'(k,\ell,m,j) ] \le n^{-\eps k}$ if $kp \le 1$, and $\Ex[ Z'(k,\ell,m,j) ] \le e^{-\sqrt{n}}$ otherwise.

  Let us fix $k$, $m$, $\ell$, and $j$, and count the triples $(S,T,U)$ which contribute to $Z'(k,\ell,m,j)$. There are at most $3^k {n \choose {k-m}}$ choices for the set $S$ if $n \not\in S$ (since $S$ intersects $k-m$ of the pairs $\{x,-x\}$), and similarly there are at most $3^k {n \choose {k-m-1}}$ choices for $S$ if $n \in S$. Also, regardless of whether $n \in S$ or $n \not\in S$, there are at most ${n \choose k}$ choices for $S$ with $|S| = k$.  
  
  Now, for each $S \subseteq E_{2n}$ and $\ell,j \in \N$, let $W(S,\ell,j)$ denote the number of pairs $(T,U)$ such that $T,U \subseteq G_p \cap O_{2n}$, $T \cap U = \emptyset$ and $T \subseteq N_{\G_S}(U)$, with $|T| = \ell$ and $|U| = j$. 

  \begin{cl}
    \label{cl:claim-three-Z2n}
    If $|S| = k$ and $0 \le j \le \ell \le k \le \eps pn$, then
    \begin{equation}
      \label{ineq:triples}
      \Ex\big[ W(S,\ell,j) \big] \, \le \, (3e^2p^2n)^k \, \ll \, (C \log n)^{2k}.
    \end{equation}
  \end{cl}

  \begin{proof}
    We have at most ${n \choose j}$ choices for $U$ and, given $S$ and $U$, there are at most ${{3kj} \choose \ell}$ choices for $T$ since $T \subseteq N_{\G_S}(U)$ and $\Delta(\G_S) \le 3k$. Since $T$ and $U$ are disjoint, the probability that $T$ and $U$ are contained in $G_p$ is $p^{\ell+j}$. 

    To simplify the computation, note that for fixed $\ell$ and $k$ with $0 \le \ell \le k \le \eps pn$, the functions ${n \choose j}p^j$ and ${3kj \choose \ell}$ are increasing in $j$ if $0 \le j \le \ell$. Therefore,
    \begin{equation}
      \label{ineq:jl}
      \Ex\big[ W(S,\ell,j) \big]  \, \le \, p^{\ell+j}{n \choose j}{3kj \choose \ell} \, \le \, p^{2\ell}{n \choose \ell}{3k\ell \choose \ell} \, \le \, \left(\frac{3e^2p^2nk}{\ell}\right)^\ell.
    \end{equation}
    Since, for any $a > 0$, the function $x \mapsto \left(\frac{a}{x}\right)^x$ is increasing for $0 \le x \le \frac{a}{e}$ and since $3e^2p^2n \gg 1$ and $0 \le \ell \le k$, the quantity in the right-hand side of \eqref{ineq:jl} is maximized when $\ell = k$. This yields \eqref{ineq:triples}.
  \end{proof}
  
  Finally, recall that if $(S, T, U)$ is good, then no edge of the graph 
  \[
  \G_{S,T,U} \, := \, \G_S\big[ O_{2n} \setminus (T \cup U) \big]
  \]
  has both its endpoints in $G_p$. Let $\S(k,m)$ denote the set of $S \subseteq E_{2n}$ with $|S| = k$ and with $m$ pairs $\{x,-x\}$ in $S$. Since the vertex set of $\G_{S,T,U}$ is disjoint from $S \cup T \cup U$, it follows that the events $e(\G_{S,T,U}[G_p]) = 0$ and $S,T,U \subseteq G_p$ are all independent. Therefore,
  \begin{equation}
    \label{Zbound}
    \Ex\big[ Z'(k,\ell,m,j) \big] \, \le \, \sum_{S \in \S(k,m)} \Pr(S \subseteq G_p) \cdot \Ex\big[ W(S,\ell,j) \big] \cdot \max_{T,U} \left\{ \Pr\Big( e\big( \G_{S,T,U}[G_p] \big) = 0 \Big) \right\},
  \end{equation}
  where the maximum is taken over all pairs $(T,U)$ as in the definition of $W(S,\ell,j)$. Hence, to complete the proof, it only remains to give a uniform bound on the probability that $e(\G_{S,T,U}[G_p]) = 0$. We shall do so using Janson's inequality. 

  Note first that, by Claim~\ref{cl:claim-one-Z2n}, $\Delta(\G_{S,T,U}) \le \Delta(\G_S) \le 3k$ and\
 \begin{equation}
    \label{eq:eGSTU}
  e(\G_{S,T,U}) \, \ge \, e(\G_S) - (|T| + |U|)\Delta(\G_S) \, = \, \left( \frac{3k - 2m - \one_S(n)}{2} \right) n - O\big( k^2 \big) \, \ge \, \frac{kn}{2},
 \end{equation}
where the final inequality follows since $2m \le k \ll n$. Set 
$$\mu \, = \, \frac{3k}{2}p^2n \qquad \text{and} \quad \quad \mu' \, = \, p^2 \cdot e(\G_{S,T,U})$$ 
and observe that $\mu/3 \le \mu' \le \mu$, by~\eqref{eq:eGSTU}, and that
   \begin{equation}
    \label{eq:deltaprime}
  k^2p^3 n \, \le \, \Delta' \, := \, \ds\sum_{B_1 \sim B_2} p^{|B_1 \cup B_2|} \, \le \, {3k \choose 2}p^3n,
   \end{equation}
  where the sum is over pairs of edges of $\G_{S,T,U}$ that share an endpoint. There are two cases.

  \bigskip
  \noindent
  {\bf Case 1. $\Delta' \le \mu'$.}
  \medskip
  
  In this case, we shall show that $\Ex[Z'(k,\ell,m,j)] \le n^{-\eps k}$. By Janson's inequality and~\eqref{eq:eGSTU}, we have
  \begin{equation}
    \label{ineq:JansonCase1}
    \Pr\Big( e\big( \G_{S,T,U}[G_p] \big) = 0\Big) \, \le \, e^{-\mu' + \Delta'/2} \, \le \, e^{-\mu + \Delta'/2} \exp\left( \frac{2m+\one_S(n)}{2} p^2n + O(p^2k^2)\right).
  \end{equation}
  Suppose first that $\Delta' \le \eps \mu$. Then $kp \le 1$, by~\eqref{eq:deltaprime}, and so  
  \begin{equation}
    \label{mudelta}
    - \mu  + \frac{\Delta'}{2} + O\big( p^2k^2 \big) \, \le \, - \big( 1 - \eps \big) \mu \, = \, -\left( 1 - \eps \right) \frac{3C}{2} k \log n \, \le \, - \left( \frac{1}{2} + 2\eps \right) k \log n,
  \end{equation}
  since $p^2 n = C \log n$ and $C > 1/3 + 2\eps$, provided that $\eps > 0$ is sufficiently small. Moreover,
  \begin{align}\label{ineq:m}
    \sum_{S \in \S(k,m)} \exp\left(\frac{2m + \one_S(n)}{2}  p^2n\right) & \, \le \, 3^k \left[{n \choose k-m} e^{p^2nm} + {n \choose k-m-1}e^{p^2n(m+1)}\right] \nonumber \\
    & \, \le \, 2^{O(k)} n^{Cm} \bigg( \frac{n}{k} \bigg)^{k-m} \big( 1 + k n^{C-1} \big)
  \end{align}
  by the discussion above Claim~\ref{cl:claim-three-Z2n}, and since\footnote{Here we used the usual bound ${n \choose k} \le \big( \frac{en}{k} \big)^k$, except when $k - m  = 1$, in which case ${n \choose k-m-1} = 1$.} $2m \le k$ and $e^{p^2n} = n^C$.  
  
 Hence, combining Claim~\ref{cl:claim-three-Z2n} with~\eqref{Zbound} and~\eqref{ineq:JansonCase1}--\eqref{ineq:m}, we obtain 
  \begin{align} \label{eq:ExZbound:case1}
 & \Ex\big[ Z'(k,\ell,m,j) \big] \, \le \, p^k \cdot (C \log n)^{2k} \sum_{S \in \S(k,m)} \max_{T,U} \Big\{ \Pr\Big( e\big( \G_{S,T,U}[G_p] \big) = 0\Big) \Big\} \nonumber\\
 & \hspace{2cm} \, \le \, \bigg( O(1) \cdot p \cdot (C \log n)^2 \cdot n^{-1/2 - 2\eps} \cdot \frac{n}{k} \bigg)^k  \left( k n^{C-1}  \right)^m \big( 1 + k n^{C-1} \big) \, \ll \, n^{-\eps k}
  \end{align}
  as $n \to \infty$, assuming that $C < 1/2$ (which implies that $k n^{C-1} < 1$, since $kp \le 1$). If $C \ge 1/2$, then a similar calculation works, but we do not need to estimate so precisely, as we can obtain a much stronger bound in \eqref{mudelta}, which allows us to replace the $n^{-1/2-2\eps}$ term in~\eqref{eq:ExZbound:case1} with  $n^{-4C/3}$; moreover, we recall that $m \le k/2$ and if $C \le 4$, then $pn \le 2\sqrt{n \log n}$. We leave the details to the reader. 

  The case $\eps\mu < \Delta' \le \mu'$ is similar, so we shall skip some of the details. Note that $3kp > \eps$, by~\eqref{eq:deltaprime}, and hence $pn/k < (3 / \eps) p^2 n = (3C / \eps) \log n$. Observe also that~\eqref{ineq:m} still holds and that 
  \[
  - \mu  + \frac{\Delta'}{2} + O\big( p^2k^2 \big) \, \le \, -\frac{2\mu}{5} \, = \, - \frac{3C}{5} k \log n,
  \]
  since $\Delta' \le \mu' \le \mu$ and $k \ll n$. Thus~\eqref{eq:ExZbound:case1} becomes in this case
  \[
  \Ex\big[ Z'(k,\ell,m,j) \big] \, \le \, \left( O(1) \cdot p \cdot (C \log n)^2 \cdot n^{-3C/5} \cdot \frac{n}{k} \right)^k  \left( k n^{C-1}  \right)^m \big( 1 + k n^{C-1} \big) \, \ll \, n^{-\eps k},
  \]
 since $m \le k/2$. It follows that if $\Delta' \le \mu'$, then $\Ex[Z'(k,\ell,m,j)] \le n^{-\eps k}$, as claimed.

  \bigskip
  \noindent
  {\bf Case 2. $\Delta' \ge \mu'$.}
  \medskip

  In this case, we shall show that $\Ex[Z'(k,\ell,m,j)] \le e^{-\sqrt{n}}$. Indeed, by Janson's inequality,
  \[
  \Pr\Big( e\big( \G_{S,T,U}[G_p] \big) = 0 \Big) \, \le \, \exp\left(-\frac{(\mu')^2}{2\Delta'}\right) \, \le \, \exp\left(-\frac{\mu^2}{18\Delta'}\right) \, \le \, e^{- pn / 36},
  \]
  by~\eqref{eq:deltaprime}, and since $\mu' \ge \frac{\mu}{3}$. We shall also need an improved version of Claim~\ref{cl:claim-three-Z2n} when $k$ is large, since the upper bound ${{3kj} \choose \ell}$ on the number of choices for $T$ becomes very bad. Fortunately, the following bound is trivial:
  \[
  \Ex[ W(S,\ell,j)] \, \le \, {n \choose \ell} {n \choose j} p^{\ell + j} \, \le \, {n \choose k}^2 p^{2k},
  \]
  where the second inequality follows since $j \le \ell \le k \le \eps p n$. Finally, recall that we have at most ${n \choose k}$ choices for $S$, and that $pn \gg \sqrt{n}$. By \eqref{Zbound}, it follows that
  \begin{align*}
  \Ex[Z'(k,\ell,m,j)] & \, \le \, p^k \cdot {n \choose k}^2 p^{2k} \sum_{S \in \S(k,m)} \max_{T,U} \Big\{ \Pr\Big( e\big( \G_{S,T,U}[G_p] \big) = 0\Big) \Big\}\\
  & \, \le \, {n \choose k}^3 p^{3k} e^{-\mu^2/(18\Delta')} \, \le \, \left(\frac{epn}{k}\right)^{3k} e^{-pn/36} \, \le \, e^{-\sqrt{n}},
  \end{align*}
  since $k \le \eps p n$, as claimed. 

  \bigskip

  Having bounded $\Ex[Z']$ in both cases, the result now follows easily by summing over $k$, $\ell$, $m$, and $j$.  Indeed, by~\eqref{PatmostZ}, together with the application of Theorem~\ref {thm:approx-stability-full} noted at the start of the proof, we have
  \begin{align*}
    \Pr\Big( \SF(G_p) \ne \{G_p \cap O_{2n}\} \Big) & \, \le \, \sum_{k=1}^{\eps p n}\sum_{m = 0}^{k/2}\sum_{\ell = 0}^k\sum_{j=0}^{\ell} \Ex\big[ Z'(k,\ell,m,j) \big] \, + \, o(1)\\
    & \, \le \, \sum_{k=1}^{\eps p n} (k+1)^3\max\left\{n^{-\eps k}, e^{-\sqrt{n}}\right\}  \, + \, o(1) \; \to \; 0
  \end{align*}
  as $n \to \infty$, as required. This completes the proof of Theorem~\ref{thm:sharp}.
\end{proof}

\bigskip
\noindent
{\bf Acknowledgment.} The bulk of this work was done during the IPAM Long Program, Combinatorics: Methods and Applications in Mathematics and Computer Science. We would like to thank the organizers of this program, the members of staff at IPAM, and all its participants for creating a fantastic research environment. We would also like to thank David Conlon, Gonzalo Fiz Pontiveros, Simon Griffiths, Tomasz {\L}uczak, and Vera S{\'o}s for helpful and stimulating discussions.

\bibliographystyle{siam}
\bibliography{BMS_sumfree}

\begin{thebibliography}{10}

\bibitem{AW}
{\sc H.~L. Abbott and E.~T.~H. Wang}, {\em Sum-free sets of integers}, Proc.
  Amer. Math. Soc., 67 (1977), pp.~11--16.

\bibitem{A}
{\sc N.~Alon}, {\em Independent sets in regular graphs and sum-free subsets of
  finite groups}, Israel J. Math., 73 (1991), pp.~247--256.

\bibitem{AlBaMoSa-SF}
{\sc N.~Alon, J.~Balogh, R.~Morris, and W.~Samotij}, {\em Counting sum-free
  sets in {A}belian groups}.
\newblock To appear in Israel Journal of Mathematics.

\bibitem{AlBaMoSa-CE}
\leavevmode\vrule height 2pt depth -1.6pt width 23pt, {\em A refinement of the
  {C}ameron-{E}rd{\H o}s conjecture}.
\newblock To appear in Proceedings of the London Mathematical Society.

\bibitem{AK}
{\sc N.~Alon and D.~J. Kleitman}, {\em Sum-free subsets}, in A tribute to
  {P}aul {E}rd{\H o}s, Cambridge Univ. Press, Cambridge, 1990, pp.~13--26.

\bibitem{AS}
{\sc N.~Alon and J.~H. Spencer}, {\em The probabilistic method},
  Wiley-Interscience Series in Discrete Math. and Optimization, John Wiley \&
  Sons Inc., Hoboken, NJ, third~ed., 2008.
\newblock With an appendix on the life and work of Paul Erd{\H{o}}s.

\bibitem{BSS}
{\sc L.~Babai, M.~Simonovits, and J.~Spencer}, {\em Extremal subgraphs of
  random graphs}, J. Graph Theory, 14 (1990), pp.~599--622.

\bibitem{BaMoSa-ind}
{\sc J.~Balogh, R.~Morris, and W.~Samotij}, {\em Independent sets in
  hypergraphs}.
\newblock arXiv:1204.6530 [math.CO].

\bibitem{BT}
{\sc B.~Bollob{\'a}s and A.~Thomason}, {\em Threshold functions},
  Combinatorica, 7 (1987), pp.~35--38.

\bibitem{BPS}
{\sc G.~Brightwell, K.~Panagiotou, and A.~Steger}, {\em Extremal subgraphs of
  random graphs}, Random Structures Algorithms, 41 (2012), pp.~147--178.

\bibitem{C}
{\sc N.~J. Calkin}, {\em On the number of sum-free sets}, Bull. London Math.
  Soc., 22 (1990), pp.~141--144.

\bibitem{Cam2}
{\sc P.~J. Cameron}, {\em Portrait of a typical sum-free set}, in Surveys in
  combinatorics 1987 ({N}ew {C}ross, 1987), vol.~123 of London Math. Soc.
  Lecture Note Ser., Cambridge Univ. Press, Cambridge, 1987, pp.~13--42.

\bibitem{CE1}
{\sc P.~J. Cameron and P.~Erd{\H{o}}s}, {\em On the number of sets of integers
  with various properties}, in Number theory ({B}anff, {AB}, 1988), de Gruyter,
  Berlin, 1990, pp.~61--79.

\bibitem{CKS}
{\sc S.~L.~G. Choi, J.~Koml{\'o}s, and E.~Szemer{\'e}di}, {\em On sum-free
  subsequences}, Trans. Amer. Math. Soc., 212 (1975), pp.~307--313.

\bibitem{CG}
{\sc D.~Conlon and T.~Gowers}, {\em Combinatorial theorems in sparse random
  sets}.
\newblock arXiv:1011.4310v1 [math.CO].

\bibitem{DMKa}
{\sc B.~DeMarco and J.~Kahn}, {\em Mantel's theorem for random graphs}.
\newblock arXiv:1206.1016 [math.PR].

\bibitem{DY}
{\sc P.~H. Diananda and H.~P. Yap}, {\em Maximal sum-free sets of elements of
  finite groups}, Proc. Japan Acad., 45 (1969), pp.~1--5.

\bibitem{EKR}
{\sc P.~Erd{\H{o}}s, D.~J. Kleitman, and B.~L. Rothschild}, {\em Asymptotic
  enumeration of {$K\sb{n}$}-free graphs}, in Colloquio {I}nternazionale sulle
  {T}eorie {C}ombinatorie ({R}ome, 1973), {T}omo {II}, Accad. Naz. Lincei,
  Rome, 1976, pp.~19--27. Atti dei Convegni Lincei, No. 17.

\bibitem{FrRo}
{\sc P.~Frankl and V.~R{\"o}dl}, {\em Large triangle-free subgraphs in graphs
  without {$K_4$}}, Graphs Combin., 2 (1986), pp.~135--144.

\bibitem{Frei}
{\sc G.~A. Freiman}, {\em On the structure and the number of sum-free sets},
  Ast\'erisque,  (1992), pp.~13, 195--201.
\newblock Journ{\'e}es Arithm{\'e}tiques, 1991 (Geneva).

\bibitem{Fried}
{\sc E.~Friedgut}, {\em Sharp thresholds of graph properties, and the {$k$}-sat
  problem}, J. Amer. Math. Soc., 12 (1999), pp.~1017--1054.
\newblock With an appendix by Jean Bourgain.

\bibitem{FK}
{\sc E.~Friedgut and G.~Kalai}, {\em Every monotone graph property has a sharp
  threshold}, Proc. Amer. Math. Soc., 124 (1996), pp.~2993--3002.

\bibitem{FRRT}
{\sc E.~Friedgut, V.~R{\"o}dl, A.~Ruci{\'n}ski, and P.~Tetali}, {\em A sharp
  threshold for random graphs with a monochromatic triangle in every edge
  coloring}, Mem. Amer. Math. Soc., 179 (2006), pp.~vi+66.

\bibitem{FRS}
{\sc E.~Friedgut, V.~R{\"o}dl, and M.~Schacht}, {\em Ramsey properties of
  random discrete structures}, Random Structures Algorithms, 37 (2010),
  pp.~407--436.

\bibitem{GRR}
{\sc R.~Graham, V.~R{\"o}dl, and A.~Ruci{\'n}ski}, {\em On {S}chur properties
  of random subsets of integers}, J. Number Theory, 61 (1996), pp.~388--408.

\bibitem{G1}
{\sc B.~Green}, {\em The {C}ameron-{E}rd{\H o}s conjecture}, Bull. London Math.
  Soc., 36 (2004), pp.~769--778.

\bibitem{G2}
\leavevmode\vrule height 2pt depth -1.6pt width 23pt, {\em A {S}zemer\'edi-type
  regularity lemma in abelian groups, with applications}, Geom. Funct. Anal.,
  15 (2005), pp.~340--376.

\bibitem{GR1}
{\sc B.~Green and I.~Z. Ruzsa}, {\em Counting sumsets and sum-free sets modulo
  a prime}, Studia Sci. Math. Hungar., 41 (2004), pp.~285--293.

\bibitem{GR2}
\leavevmode\vrule height 2pt depth -1.6pt width 23pt, {\em Sum-free sets in
  abelian groups}, Israel J. Math., 147 (2005), pp.~157--188.

\bibitem{Hatami}
{\sc H.~Hatami}, {\em A structure theorem for {B}oolean functions with small
  total influences}.
\newblock To appear in Annals of Mathematics.

\bibitem{KoLeRoSa}
{\sc Y.~Kohayakawa, S.~Lee, V.~R{\"o}dl, and W.~Samotij}, {\em The number of
  {S}idon sets and the maximum size of {S}idon sets contained in a sparse
  random set of integers}.
\newblock To appear in Random Structures \& Algorithms.

\bibitem{KLR1}
{\sc Y.~Kohayakawa, T.~{\L}uczak, and V.~R{\"o}dl}, {\em Arithmetic
  progressions of length three in subsets of a random set}, Acta Arith., 75
  (1996), pp.~133--163.

\bibitem{KSV}
{\sc D.~Kr{\'a}l, O.~Serra, and L.~Vena}, {\em A combinatorial proof of the
  removal lemma for groups}, J. Combin. Theory Ser. A, 116 (2009),
  pp.~971--978.

\bibitem{LLS}
{\sc V.~F. Lev, T.~{\L}uczak, and T.~Schoen}, {\em Sum-free sets in abelian
  groups}, Israel J. Math., 125 (2001), pp.~347--367.

\bibitem{LS}
{\sc V.~F. Lev and T.~Schoen}, {\em Cameron-{E}rd{\H o}s modulo a prime},
  Finite Fields Appl., 8 (2002), pp.~108--119.

\bibitem{Luc}
{\sc T.~{\L}uczak}, {\em Randomness and regularity}, in International
  {C}ongress of {M}athematicians. {V}ol. {III}, Eur. Math. Soc., Z\"urich,
  2006, pp.~899--909.

\bibitem{RR1}
{\sc V.~R{\"o}dl and A.~Ruci{\'n}ski}, {\em Threshold functions for {R}amsey
  properties}, J. Amer. Math. Soc., 8 (1995), pp.~917--942.

\bibitem{RR2}
\leavevmode\vrule height 2pt depth -1.6pt width 23pt, {\em Rado partition
  theorem for random subsets of integers}, Proc. London Math. Soc. (3), 74
  (1997), pp.~481--502.

\bibitem{Roth}
{\sc K.~F. Roth}, {\em On certain sets of integers}, J. London Math. Soc., 28
  (1953), pp.~104--109.

\bibitem{Sa}
{\sc W.~Samotij}, {\em Stability results for random discrete structures}.
\newblock To appear in Random Structures \& Algorithms.

\bibitem{Sa-CE}
{\sc A.~A. Sapozhenko}, {\em The {C}ameron-{E}rd{\H o}s conjecture}, Dokl.
  Akad. Nauk, 393 (2003), pp.~749--752.

\bibitem{SaTh}
{\sc D.~Saxton and A.~Thomason}, {\em Hypergraph containers}.
\newblock arXiv:1204.6595 [math.CO].

\bibitem{Sch}
{\sc M.~Schacht}, {\em Extremal results for random discrete structures}.
\newblock submitted.

\bibitem{Schur}
{\sc I.~Schur}, {\em {\"U}ber die kongruenz $x^m + y^m \equiv z^m \pmod p$},
  Jber. Deutsch. Math.-Verein., 25 (1917), pp.~114--117.

\bibitem{Yap69}
{\sc H.~P. Yap}, {\em Maximal sum-free sets of group elements}, J. London Math.
  Soc., 44 (1969), pp.~131--136.

\bibitem{Yap70}
{\sc H.~P. Yap}, {\em Structure of maximal sum-free sets in groups of order
  {$3p$}}, Proc. Japan Acad., 46 (1970), pp.~758--762.

\end{thebibliography}
 
\end{document}